\documentclass[a4paper,11pt]{article}

\usepackage[utf8]{inputenc}
\usepackage[UKenglish]{babel}
\usepackage[T1]{fontenc}
\usepackage{amsmath,amsfonts,amssymb,amsthm}
\usepackage{graphicx}
\usepackage[top=2cm,bottom=3cm,left=2.5cm,right=2.5cm]{geometry}
\usepackage{setspace}
\usepackage{xcolor}
\usepackage{stmaryrd}
\usepackage{enumitem}
\usepackage{esint}
\usepackage{centernot}
\usepackage[format=hang,singlelinecheck=false]{caption}
\usepackage{framed}
\usepackage[normalem]{ulem}
\usepackage{cancel}

\usepackage[final,colorlinks=true]{hyperref}

\author{Guillaume Blanc\thanks{École Polytechnique Fédérale de Lausanne, \href{mailto:guillaume.blanc@epfl.ch}{guillaume.blanc@epfl.ch}; \url{https://sites.google.com/view/guillaume-blanc-math}} \and Alice Contat\thanks{Université Sorbonne Paris Nord, \href{mailto:alice.contat@math.cnrs.fr}{alice.contat@math.cnrs.fr}; \url{https://sites.google.com/view/acontat/}}}
\title{Random burning of the Euclidean torus}
\date{}

\begin{document}

\theoremstyle{plain}
\newtheorem{thm}{Theorem}
\newtheorem*{thm*}{Theorem}
\newtheorem{prop}{Proposition}
\newtheorem{lem}{Lemma}
\newtheorem{cor}{Corollary}
\newtheorem{claim}{Claim}
\newtheorem{conj}{Conjecture}

\theoremstyle{remark}
\newtheorem{rem}{Remark}

\maketitle

\begin{abstract}
The burning number of a graph is the minimal number of steps that are needed to burn all of its vertices, with the following burning procedure: at each step, one can choose a point to set on fire, and the fire propagates constantly at unit speed along the edges of the graph.
In this paper, we consider two natural \emph{random} burning procedures in the discrete Euclidean torus $\mathbb{T}_n^d$, in which the points that we set on fire at each step are random variables.
Our main result deals with the case where at each step, the law of the new point that we set on fire conditionally on the past is the uniform distribution on the complement of the set of vertices burned by the previous points.
In this case, we prove that as $n\to\infty$, the corresponding random burning number (i.e, the first step at which the whole torus is burned) is asymptotic to $T\cdot n^{d/(d+1)}$ in probability, where $T=T(d)\in(0,\infty)$ is the explosion time of a so-called generalised Blasius equation.
\end{abstract}

\section*{Introduction and main results}
The \emph{burning number} $b(G)$ of a finite graph $G$ is the minimal number of steps that are needed to burn all of its vertices, with the following burning procedure: at each step $i\in\mathbb{N}^*=\{1,2,\ldots\}$,\footnote{Throughout the paper, we use the notation $\mathbb{N}=\{0,1,\ldots\}$ and $\mathbb{N}^*=\{1,2,\ldots\}$, and $\mathbb{R}_+^*={]0,\infty[}$.} one can choose a vertex $x_i\in G$ to set on fire, and the fire propagates constantly at unit speed along the edges of the graph.
Formally, we have
\[b(G):=\inf\left\{k\in\mathbb{N}^*:\text{there exists $x_1,\ldots,x_k\in G$ such that $\bigcup_{i=1}^k\overline{B}(x_i,k-i)=G$}\right\},\]
where $ \overline{B}(x,r)$ denotes the closed ball of radius $r$ around the vertex $x$ in the graph $G$.
This graph parameter first appeared in a paper by Alon \cite{alon}, who presented it as a communication problem originating from Intel, and was reintroduced more recently under that name by Bonato, Janssen and Roshanbin \cite{bonatojanssenroshanbin}, with motivation arising from social network analysis.
It has then received a lot of attention from the graph-theoretical point of view, in particular due to the resistance of the so-called burning number conjecture \cite{bessybonatojanssenrautenbachroshanbin,landlu,bastidebonamybonatocharbitkamalipierronrabie,norinturcotte}.
On the probabilistic side, the burning number was first considered by Mitsche, Pra\l{}at and Roshanbin \cite{mitschepralatroshanbin}, who in particular estimated the burning number of some large random graphs; in the same vein, let us also mention the recent work of Devroye, Eide and Pra\l{}at \cite{devroyeeidepralat}, who estimated the burning number of large Bienaymé trees.
In addition, Mitsche, Pra\l{}at and Roshanbin \cite{mitschepralatroshanbin} studied the asymptotic behaviour of some natural random burning procedures (on path graphs): \emph{this is the focus of our paper}.
These random burning procedures can be thought of as greedy constructions of nearly-optimal burning sequences.
Such greedy procedures are often mathematically interesting to analyse (see, e.g, \cite{greedylocal} for the case of the independence number), and provide a bound on the corresponding graph parameter (in our case, an upper bound on the burning number) which often turns out to be of the same order of magnitude.
Our main result (Theorem \ref{thm:rejectionsampling} below) treats the case of the burning number in the discrete Euclidean torus, thereby complementing \cite[Theorem 1.4.(ii)]{mitschepralatroshanbin}.

\paragraph{Burning number of the discrete Euclidean torus.}
Fix $d\in\mathbb{N}^*$.
For each ${n\in\mathbb{N}^*}$, we denote by $\mathbb{T}_n^d$ the discrete Euclidean torus $\left.\mathbb{Z}^d\middle/n\mathbb{Z}^d\right.$ with side length $n$, equipped with its natural graph structure.
In dimension $d=1$, the burning number of the discrete Euclidean torus was first calculated by Bonato, Janssen and Roshanbin \cite[Corollary 2.10]{bonatojanssenroshanbin}, who showed that ${b\left(\mathbb{T}_n^1\right)=\left\lceil\sqrt{n}\right\rceil}$.
In dimension $d=2$, the burning number of the discrete Euclidean torus was estimated by Mitsche, Pra\l{}at and Roshanbin
\cite[Corollary 4.1]{mitschepralatroshanbin}, who proved that ${b\left(\mathbb{T}_n^2\right)\sim\left(3/2\right)^{1/3}\cdot n^{2/3}}$ as $n\to\infty$.
In general dimension $d\in\mathbb{N}^*$, it is not difficult to show that as $n\to\infty$, the burning number of $\mathbb{T}_n^d$ is of order $n^{d/(d+1)}$:

\begin{prop}\label{prop:detbounds}
As $n\to\infty$, we have
\[\left(\alpha_d+o(1)\right)\cdot n^{d/(d+1)}\leq b\left(\mathbb{T}_n^d\right)\leq\left(\gamma_d+o(1)\right)\cdot n^{d/(d+1)},\]
where
\[\alpha_d=\left(\frac{(d+1)!}{2^d}\right)^{1/(d+1)}\quad\text{and}\quad\gamma_d=\left(\left(1+\frac{1}{d}\right)^d\cdot(d+1)!\right)^{1/(d+1)}.\]
\end{prop}

The proof of this proposition, which we present in Section \ref{sec:detbounds}, is based on volume and covering arguments.
We leave open the more difficult question of obtaining a precise asymptotic.
Note that in dimensions $d=1,2$, the lower bound in Proposition \ref{prop:detbounds} is asymptotically sharp, i.e, we have $b\left(\mathbb{T}_n^d\right)\sim\alpha_d\cdot n^{d/(d+1)}$ as $n\to\infty$.
Moreover, observe that as $d\to\infty$, we have $\alpha_d\sim d/(2\mathrm{e})$ and $\gamma_d\sim d/\mathrm{e}$.

\paragraph{Random burning of the discrete Euclidean torus.}
Given a sequence $(x_1,x_2,\ldots)$ in $\mathbb{T}_n^d$, consider the burning process in which at each step $i\in\mathbb{N}^*$, we set the point $x_i$ on fire.
For each $k\in\mathbb{N}$, the set of vertices burned by $x_1,\ldots,x_k$ at step $k$ of the process is
\[\mathsf{Burned}(x_1,\ldots,x_k):=\bigcup_{i=1}^k\overline{B}(x_i,k-i).\]
In this paper, following Mitsche, Pra\l{}at and Roshanbin \cite{mitschepralatroshanbin}, we consider the burning process ${(\mathsf{Burned}(X_1,\ldots,X_k))_{k\in\mathbb{N}}}$ associated with some sequence $(X_1,X_2,\ldots)$ of random variables with values in $\mathbb{T}_n^d$, and we look at the asymptotic behaviour (as $n\to\infty$) of the corresponding random burning number:
\[\tau_n=\inf\left\{k\in\mathbb{N}:\bigcup_{i=1}^k\overline{B}(X_i,k-i)=\mathbb{T}_n^d\right\}.\]
The distribution of the random burning number $\tau_n$ obviously depends on the distribution of the sequence $(X_1,X_2,\ldots)$, and we consider two natural choices for the distribution of that sequence.
Our results are stated in Theorem \ref{thm:couponcollector} and Theorem \ref{thm:rejectionsampling} below.

\subparagraph{Coupon collector.}
Perhaps the most natural choice for the distribution of the sequence of random variables that we set on fire at each step is the following one:
Let $(X_1,X_2,\ldots)$ be independent and have uniform distribution on $\mathbb{T}_n^d$.
We denote by $\mathcal{B}_n^\mathsf{cc}(\cdot)$ the burning process associated with the sequence $(X_1,X_2,\ldots)$, and by $\tau_n^\mathsf{cc}$ the corresponding random burning number: we have
\[\mathcal{B}_n^\mathsf{cc}(k)=\mathsf{Burned}(X_1,\ldots,X_k)=\bigcup_{i=1}^k\overline{B}(X_i,k-i)\quad\text{for all $k\in\mathbb{N}$,}\]
and
\[\tau_n^\mathsf{cc}=\inf\left\{k\in\mathbb{N}:\mathcal{B}_n^\mathsf{cc}(k)=\mathbb{T}_n^d\right\}.\]
We will refer to $\mathcal{B}_n^\mathsf{cc}(\cdot)$ as the ``coupon collector'' burning process, which is what the superscript $\mathsf{cc}$ stands for.
Indeed, determining the asymptotic behaviour of $\tau_n^\mathsf{cc}$ as $n\to\infty$ can be seen as a generalisation with geometry of the standard coupon collector problem: see \cite{aldous} and \cite{camposteixeira} for other related models.
For instance, in \cite[Section 2]{aldous}, Aldous considers a collection $(X_1,X_2,\ldots)$ of independent random variables with distribution $\mu$ on a compact metric space $(S,d)$, and studies the cover time
\[C=\inf\left\{k\in\mathbb{N}:\bigcup_{i=1}^k\overline{B}(X_i,r)=S\right\},\]
where $r\in\mathbb{R}_+^*$ is fixed.
Note that when $S$ is a finite set equipped with the discrete distance and $r\in{]0,1[}$, the cover times $C$ agrees with the usual coupon collector stopping time.

Back to the coupon collector random burning number of $\mathbb{T}_n^d$ now: In dimension $d=1$, it was shown by Mitsche, Pra\l{}at and Roshanbin \cite[Theorem 1.4.(i)]{mitschepralatroshanbin} that as $n\to\infty$, we have
\[\left(\frac{n\ln n}{2}\right)^{-1/2}\cdot\tau_n^\mathsf{cc}\longrightarrow1\quad\text{in probability.}\]
We obtain the analogous result in general dimension $d\in\mathbb{N}^*$.

\begin{thm}\label{thm:couponcollector}
As $n\to\infty$, we have
\[\left(\frac{d!}{2^d}\cdot n^d\cdot\ln\left(n^d\right)\right)^{-1/(d+1)}\cdot\tau_n^\mathsf{cc}\longrightarrow1\quad\text{in probability.}\]
\end{thm}

The proof, which we present in Section \ref{sec:couponcollector}, is based on a first and second moment argument.

\subparagraph{Rejection sampling.}
Another natural choice for the distribution of the sequence of random variables that we set on fire at each step, say $(Y_1,Y_2,\ldots)$ in this case, is one such that the following condition is satisfied, say (RS): for each ${k\in\mathbb{N}}$, conditionally on $Y_1,\ldots,Y_k$, the random variable $Y_{k+1}$ has uniform distribution on the complement of $\mathsf{Burned}(Y_1,\ldots,Y_k)$, provided that $\mathsf{Burned}(Y_1,\ldots,Y_k)\varsubsetneq\mathbb{T}_n^d$.
We denote by $\mathcal{B}_n^\mathsf{rs}(\cdot)$ the burning process associated with the sequence $(Y_1,Y_2,\ldots)$, and by $\tau_n^\mathsf{rs}$ the corresponding random burning number: we have
\[\mathcal{B}_n^\mathsf{rs}(k)=\mathsf{Burned}(Y_1,\ldots,Y_k)=\bigcup_{i=1}^k\overline{B}(Y_i,k-i)\quad\text{for all $k\in\mathbb{N}$,}\]
and
\[\tau_n^\mathsf{rs}=\inf\left\{k\in\mathbb{N}:\mathcal{B}_n^\mathsf{rs}(k)=\mathbb{T}_n^d\right\}.\]
See Figure \ref{fig:burn} for an illustration.

\begin{figure}[!h]
\begin{center}
\vspace{-2cm}
\includegraphics[width=5.3cm,angle=45]{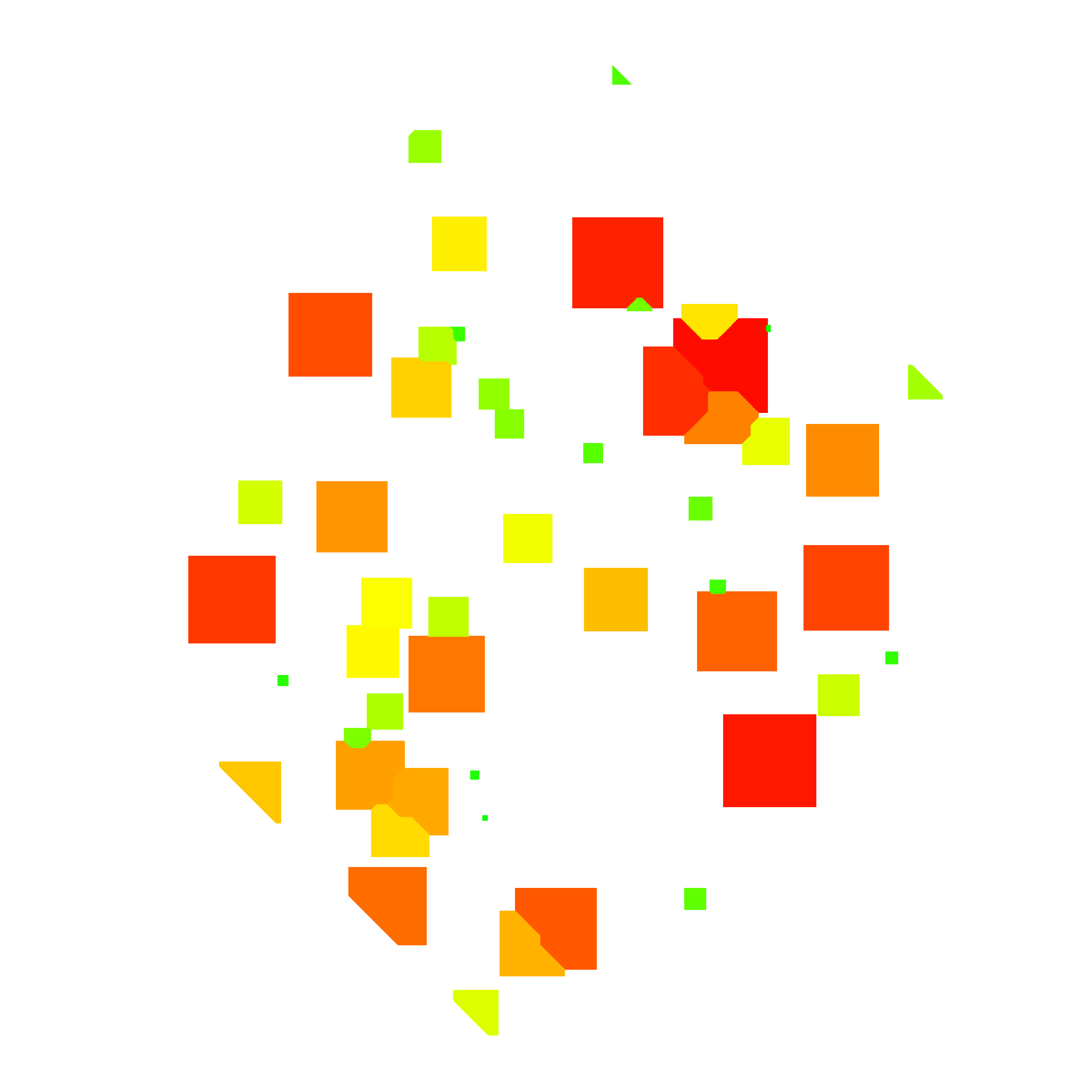} \hspace{-3.5cm}\includegraphics[width=5.3cm,angle=45]{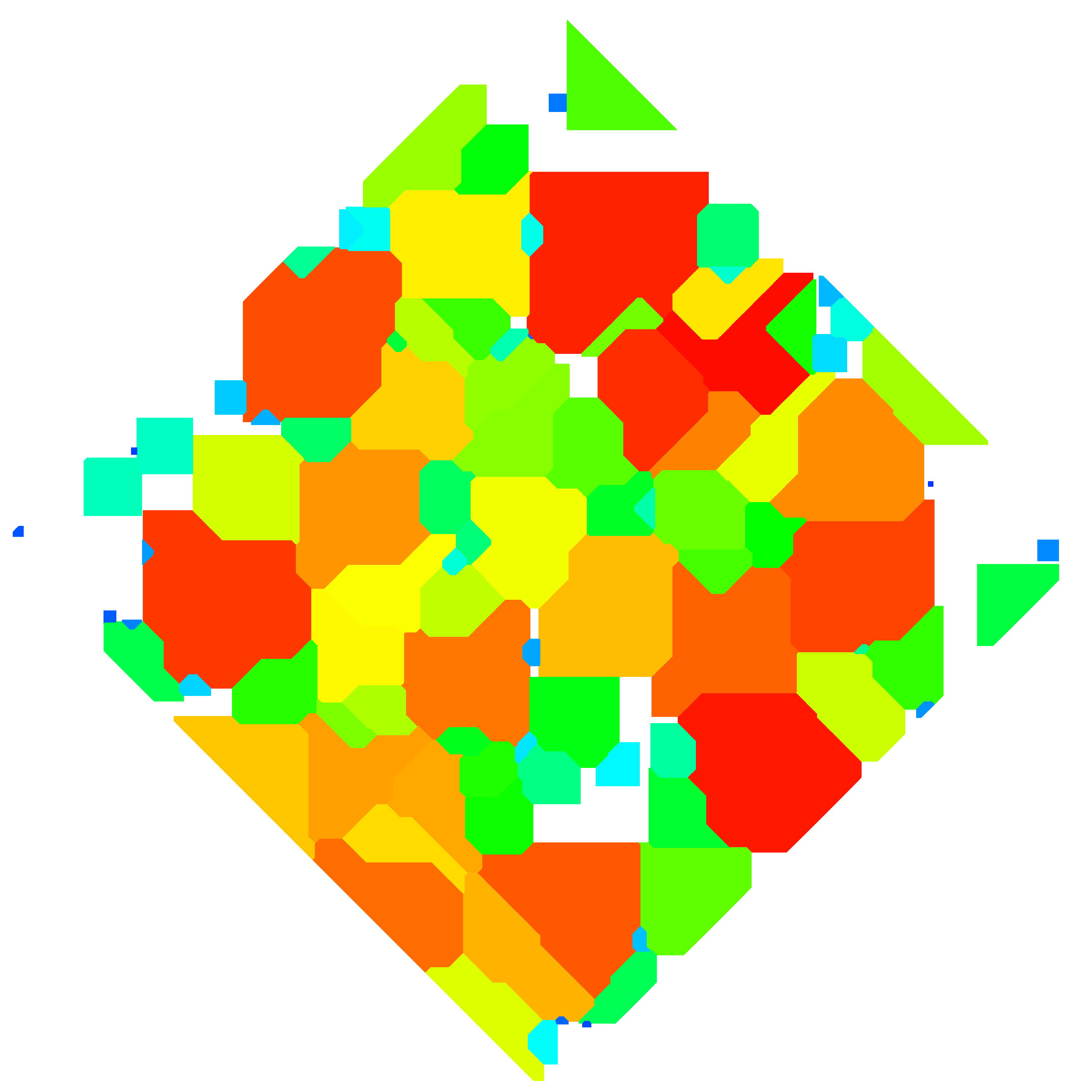} \hspace{-3.5cm}\includegraphics[width=5.3cm,angle=45]{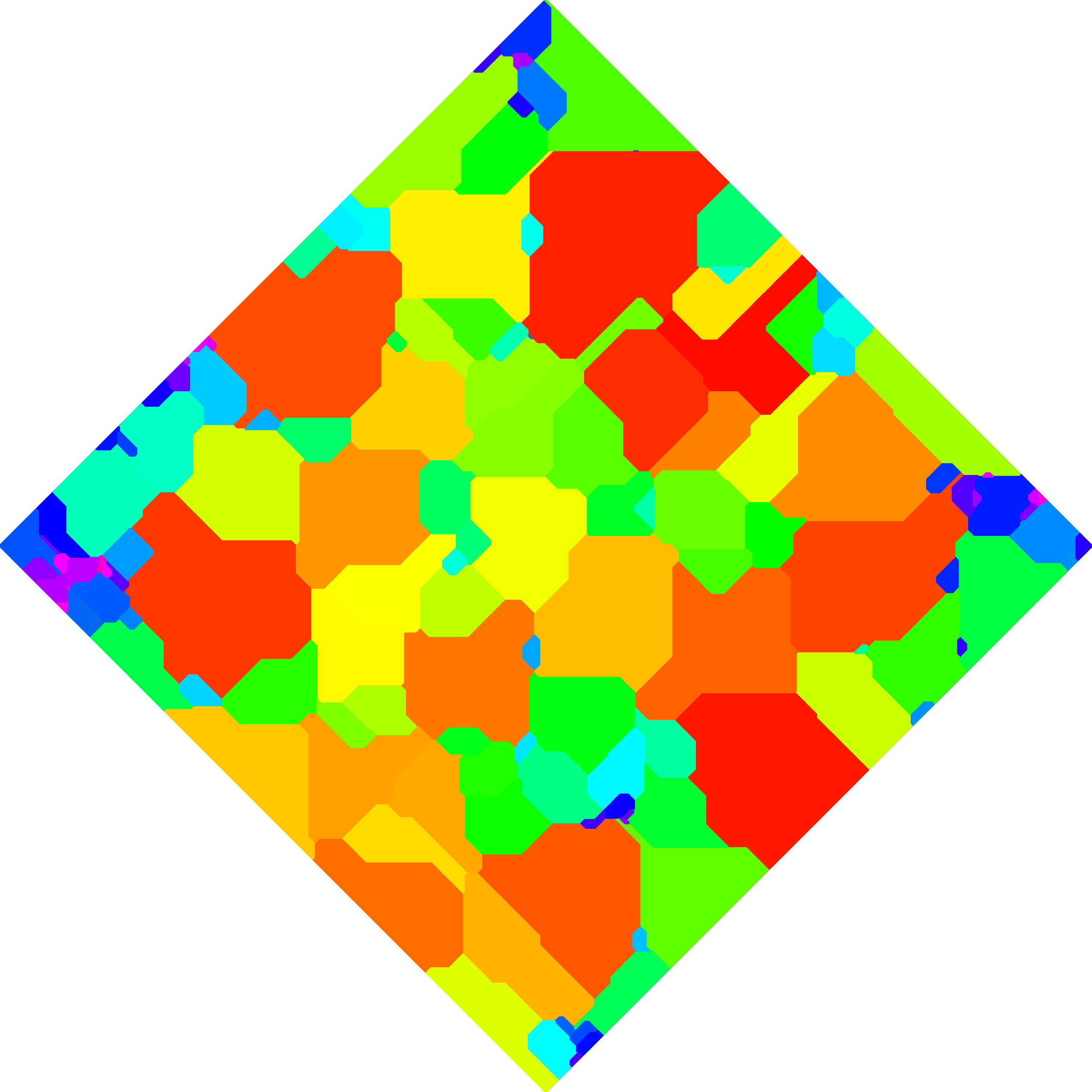}\vspace{-2cm}
\caption{\label{fig:burn}Simulation of the rejection sampling burning process on the $600\times600$ grid, at steps $50$ (left), $100$ (middle) and $147$ (right); where in the last case, the whole grid is burned. The colors of the vertices represent the arrival time of ``the'' point that burned them, from red to purple.}
\end{center}
\end{figure}

Note that even though the distribution of the sequence $(Y_1,Y_2,\ldots)$ is not uniquely determined by condition (RS), since it does not prescribe the conditional distribution of $Y_{k+1}$ given $Y_1,\ldots,Y_k$ on the event $\left(\mathsf{Burned}(Y_1,\ldots,Y_k)=\mathbb{T}_n^d\right)$, this affects neither the distribution of the random burning number $\tau_n^\mathsf{rs}$ nor that of the burning process $\mathcal{B}_n^\mathsf{rs}(\cdot)$, for as soon as ${\mathsf{Burned}(Y_1,\ldots,Y_k)=\mathbb{T}_n^d}$, we have $\mathsf{Burned}(Y_1,\ldots,Y_k,Y_{k+1})=\mathbb{T}_n^d$ regardless of $Y_{k+1}$.
Moreover, a natural way to construct a sequence $(Y_1,Y_2,\ldots)$ satisfying condition (RS) is to define it out of a collection of independent random variables with uniform distribution on $\mathbb{T}_n^d$, by using rejection sampling (see Subsection \ref{subsec:coupling}).
For this reason, we refer to $\mathcal{B}_n^\mathsf{rs}(\cdot)$ as the ``rejection sampling'' burning process, which is what the superscript $\mathsf{rs}$ stands for.
In dimension $d=1$, it was shown by Mitsche, Pra\l{}at and Roshanbin \cite[Theorem 1.4(ii)]{mitschepralatroshanbin} that there exists a constant $\gamma\in(1,\infty)$ such that
\begin{equation}\label{eq:costdrunk}
\mathbb{P}\left(\sqrt{n}\leq\tau_n^\mathsf{rs}\leq\gamma\cdot\sqrt{n}\right)\longrightarrow1\quad\text{as $n\to\infty$.}
\end{equation}

Our main result identifies the precise asymptotic of the rejection sampling random burning number of $\mathbb{T}_n^d$ in general dimension $d\in\mathbb{N}^*$, together with the scaling limit of the process $\#\mathcal{B}_n^\mathsf{rs}(\cdot)$ (which turns out to be more convenient to state for the complement).
These asymptotics involve the maximal solution $y:{[0,T[}\rightarrow\mathbb{R}$ of the \emph{generalised Blasius equation}
\begin{equation}\label{eq:blasius}
\begin{cases}
y^{(d+1)}(t)=2^d\cdot y(t)\cdot y^{(d)}(t),\quad t\in\mathbb{R}_+,\\
\text{$y(0),\ldots,y^{(d-1)}(0)=0$ and $y^{(d)}(0)=1$.}
\end{cases}
\end{equation}
Note that by the Cauchy--Lispchitz theorem, the Cauchy problem \eqref{eq:blasius} admits a unique maximal solution $y:{[0,T[}\rightarrow\mathbb{R}$, where the explosion time $T\in{]0,\infty]}$ is defined implicitly by the fact that $y:{[0,T[}\rightarrow\mathbb{R}$ is the maximal solution of \eqref{eq:blasius}.

The differential equation \eqref{eq:blasius} is related to the Blasius equation from boundary layer theory, which was first solved mathematically by Weyl \cite{weyl}.
A generalisation of the Blasius equation similar to \eqref{eq:blasius} was also considered by Weyl \cite[Section 3]{weyl}, and a close relative of \eqref{eq:blasius} comes up in the work of Kuba and Panholzer \cite[Subsection 2.6, Equation (5)]{kubapanholzer}, in counting $(d+1)$-labelled increasing trees.
The connection between the generalised Blasius equation \eqref{eq:blasius} and random burning models was unveiled in \cite{blasius}, and is explained in the heuristic below.
For now, let us state our main result, which concerns the scaling limit of the rejection sampling burning process in $\mathbb{T}_n^d$.
First, we note that by \cite[Theorem 1]{blasius}, the maximal solution $y:{[0,T[}\rightarrow\mathbb{R}$ of \eqref{eq:blasius} blows up in finite time, i.e, we have $T\in(0,\infty)$.\footnote{The multiplicative constant $2^d$ is not present in the generalised Blasius equation considered in \cite{blasius}, but the two solutions can be easily related: if $x:{[0,S[}\rightarrow\mathbb{R}$ denotes the maximal solution of the Cauchy problem $x^{(d+1)}(t)=x(t)\cdot x^{(d)}(t)$ for $t\in\mathbb{R}_+$, with initial conditions $x(0),\ldots,x^{(d-1)}(0)=0$ and $x^{(d)}(0)=1$ (where $S\in{]0,\infty]}$ is defined implicitly by the fact that $x:{[0,S[}\rightarrow\mathbb{R}$ is the maximal solution of this Cauchy problem), then we have $T=\beta\cdot S$, and $y(t)=\beta^d\cdot x\left(\beta^{-1}\cdot t\right)$ for all $t\in[0,T[$, where $\beta=2^{-d/(d+1)}$.}

\begin{thm}\label{thm:rejectionsampling}
As $n\to\infty$, we have
\begin{equation}\label{eq:asymptoticrandomburning}
\frac{\tau_n^\mathsf{rs}}{n^{d/(d+1)}}\longrightarrow T\quad\text{in probability,}
\end{equation}
where $T=T(d)\in(0,\infty)$ is the explosion time of the maximal solution $y:{[0,T[}\rightarrow\mathbb{R}$ to the Cauchy problem \eqref{eq:blasius}.
Moreover, we have
\begin{equation}\label{eq:functionalconvergence}
\sup_{t\in[0,T[}\left|\frac{\#\left(\mathbb{T}_n^d\middle\backslash\mathcal{B}_n^\mathsf{rs}\left(\left\lfloor t\cdot n^{d/(d+1)}\right\rfloor\right)\right)}{n^d}-\frac{1}{y^{(d)}(t)}\right|\overset{\mathbb{P}}{\longrightarrow}0\quad\text{as $n\to\infty$.}
\end{equation}
\end{thm}

This result is the core of the paper; we present it in Section \ref{sec:rejectionsampling}.
Note that in dimension $d=1$, the Cauchy problem \eqref{eq:blasius} can be solved explicitly: we find that $T=\pi/2$, and $y(t)=\tan(t)$ for all $t\in{[0,\pi/2[}$.
In particular, we get that
\[\frac{\tau_n^\mathsf{rs}}{\sqrt{n}}\underset{n\to\infty}{\overset{\mathbb{P}}{\longrightarrow}}\frac{\pi}{2}\quad\text{and}\quad\sup_{t\in[0,\pi/2[}\left|\frac{\#\left(\mathbb{T}_n^1\middle\backslash\mathcal{B}_n^\mathsf{rs}\left(\left\lfloor t\cdot\sqrt{n}\right\rfloor\right)\right)}{n}-\cos(t)^2\right|\underset{n\to\infty}{\overset{\mathbb{P}}{\longrightarrow}}0.\]
It is natural to compare $T = T(d)$ with the upper bound $ \gamma_d$ which appears in Proposition \ref{prop:detbounds}.
Estimating $T(d)$ is a difficult problem in general.
As far as we know, the best bounds for $T(d)$ are obtained in \cite[p.391]{weyl}: Weyl's bounds on the radius of convergence $R$ therein imply that
\[\frac{d}{2\mathrm{e}}\cdot(1+o(1))\leq T(d)\leq\frac{2d}{\mathrm{e}}\cdot(1+o(1))\quad\text{as $d\to\infty$,}\]
which does not beat the upper bound on the burning number of $\mathbb{T}_n^d$ given by $\gamma_d\sim d/\mathrm{e}$ as $d\to\infty$.

\paragraph{Heuristic for the appearance of \eqref{eq:blasius}.}
To conclude this introduction, let us explain the connection between the generalised Blasius equation \eqref{eq:blasius} and random burning models.
Consider the following model in the continuous Euclidean torus $\mathbb{T}^d=\left.\mathbb{R}^d\middle\backslash\mathbb{Z}^d\right.$, which will be designed in order to mimic the rejection sampling burning process in the discrete Euclidean torus $\mathbb{T}_n^d$ seen from distance $n$.
For each $n\in\mathbb{N}^*$, fix a continuous function $f_n:{[0,T_n[}\rightarrow[1,\infty[$ to be adjusted later, and let $\Pi_n$ be a Poisson process with intensity $\mathrm{d}x\otimes f_n(t)\mathrm{d}t$ on $\mathbb{T}^d\times[0,T_n[$.
For each atom $(x,t)$ of $\Pi_n$, imagine that we set the point $x\in\mathbb{T}^d$ on fire at time $t\in[0,T_n[$, and that the fire propagates constantly at speed $1/n$, with respect to the metric on $\mathbb{T}^d$ induced by the $1$-norm $\|\cdot\|_1$ on $\mathbb{R}^d$ (which is the scaling limit of the graph distance on $\mathbb{T}_n^d$).
For each $t\in[0,T_n[$, the set of vertices that are burned by time $t$ in this continuous burning process is
\[\mathcal{B}_n(t)=\bigcup_{\substack{(x,s)\in\Pi_n\\s\leq t}}\overline{B}\left(x,\frac{t-s}{n}\right).\]
Therefore, the mean density of unburned vertices is
\[\mathbb{E}\left[\left|\mathbb{T}^d\middle\backslash\mathcal{B}_n(t)\right|\right]=\int_{\mathbb{T}^d}\mathbb{P}\left(z\notin\mathcal{B}_n(t)\right)\mathrm{d}z=\exp\left(-\frac{2^d}{d!}\cdot\int_0^t\left(\frac{t-s}{n}\right)^d\cdot f_n(s)\mathrm{d}s\right).\]
Indeed, for all $z\in\mathbb{T}^d$, we have\footnote{For the last equality, we assume that $T_n/n\leq1/2$, so that the ball $\overline{B}(z,(t-s)/n)$ has the same Lebesgue measure as a $1$-norm ball with radius $(t-s)/n$ in $\mathbb{R}^d$.}
\begin{eqnarray*}
\mathbb{P}\left(z\notin\mathcal{B}_n(t)\right)&=&\mathbb{P}\left(\Pi_n\left(\left\{(x,s)\in\mathbb{T}^d\times[0,t]:z\in\overline{B}\left(x,\frac{t-s}{n}\right)\right\}\right)=0\right)\\
&=&\exp\left(-\int_{\mathbb{T}^d}\int_0^t\mathbf{1}\left(z\in\overline{B}\left(x,\frac{t-s}{n}\right)\right)\cdot f_n(s)\mathrm{d}s\mathrm{d}x\right)\\
&=&\exp\left(-\int_0^t\left|\overline{B}\left(z,\frac{t-s}{n}\right)\right|\cdot f_n(s)\mathrm{d}s\right)=\exp\left(-\frac{2^d}{d!}\cdot\int_0^t\left(\frac{t-s}{n}\right)^d\cdot f_n(s)\mathrm{d}s\right).
\end{eqnarray*}
Finally, we set $T_n=T\cdot n^{d/(d+1)}$, and we let $f_n(t)=y^{(d)}\left(t\cdot n^{-d/(d+1)}\right)$ for all $t\in[0,T_n[$.
Since
\[y^{(d)}(t)=\exp\left(\frac{2^d}{d!}\cdot\int_0^t(t-s)^d\cdot y^{(d)}(s)\mathrm{d}s\right)\quad\text{for all $t\in[0,T[$}\]
(see Claim \ref{claim:picard} below for a formal proof), we get that
\[f_n(t)\cdot\mathbb{E}\left[\left|\mathbb{T}^d\middle\backslash\mathcal{B}_n(t)\right|\right]=f_n(t)\cdot\exp\left(-\frac{2^d}{d!}\cdot\int_0^t\left(\frac{t-s}{n}\right)^d\cdot f_n(s)\mathrm{d}s\right)=1\quad\text{for all $t\in[0,T_n[$.}\]
Therefore, on average, the Poisson process $\Pi_n$ puts one point per unit of time in the unburned region $\left.\mathbb{T}^d\middle\backslash\mathcal{B}_n(t)\right.$.
The corresponding burning process in $\mathbb{T}^d$ is thus a continuous analogue of the rejection sampling burning process in $\mathbb{T}_n^d$.
To draw the analogy with Theorem \ref{thm:rejectionsampling}, note that the duration of this continuous burning process in $\mathbb{T}^d$ is $T_n=T\cdot n^{d/(d+1)}$, and that for all $t\in[0,T[$, the mean density of unburned vertices at time $t\cdot n^{d/(d+1)}$ is
\[\mathbb{E}\left[\left|\mathbb{T}^d\middle\backslash\mathcal{B}_n\left(t\cdot n^{d/(d+1)}\right)\right|\right]=\frac{1}{y^{(d)}(t)}.\]

\paragraph{Acknowledgements.}
We warmly thank Nicolas Curien for encouraging this collaboration and for stimulating discussions at the early stages of this project.
We acknowledge support from ``SuPerGRandMa'', the ERC Consolidator grant no.~101087572, and benefitted from a PEPS JCJC grant.
G.B.~is grateful to the members of LAGA at Université Sorbonne Paris Nord for their hospitality, and acknowledges support from SNSF Eccellenza grant 194648 and NCCR SwissMAP.
Finally, we thank the anonymous referees for their careful reading of the paper, and for their indulgence in pointing out an important but fixable mistake in a previous version of the paper.

\[***\]
\emph{Throughout the paper, we use the notation $\mathbb{N}=\{0,1,\ldots\}$ and $\mathbb{N}^*=\{1,2,\ldots\}$, and $\mathbb{R}_+^*={]0,\infty[}$.
All our random variables are implicitly defined on a probability space that we denote by $(\Omega,\mathcal{F},\mathbb{P})$.}
\[***\]

\section{Deterministic estimates}\label{sec:detbounds}

As a warmup, we start with the proof of Proposition \ref{prop:detbounds}.
Throughout the paper, we will use implicitly the following facts:
\begin{itemize}
\item For all $r\in\mathbb{N}$ such that $r<n/2$, we have
\[\#\overline{B}(0,r)=\#\left\{x\in\mathbb{Z}^d:\|x\|_1\leq r\right\},\]
where $\overline{B}(x,r)$ denotes the closed ball with radius $r$ around $x$ in the graph $\mathbb{T}_n^d$.
Note that with a slight abuse of notation, we omit the dependence in $n$: this will not cause any problem, since typically we will look at balls in $\mathbb{T}_n^d$ with radius $r<n/2$, so that there are no boundary effects and one can think about the same ball in $\mathbb{Z}^d$.

\item As $r\to\infty$, we have
\[\#\left\{x\in\mathbb{Z}^d:\|x\|_1\leq r\right\}\sim\lambda\left(\left\{x\in\mathbb{R}^d:\|x\|_1\leq1\right\}\right)\cdot r^d=\frac{2^d}{d!}\cdot r^d.\]
\end{itemize}

\begin{proof}[Proof of Proposition \ref{prop:detbounds}]
\emph{Lower bound.}
Set $k=b\left(\mathbb{T}_n^d\right)$ for short.
By definition, there exists $x_1,\ldots,x_k\in\mathbb{T}_n^d$ such that $\bigcup_{i=1}^k\overline{B}(x_i,k-i)=\mathbb{T}_n^d$.
We deduce that
\[n^d=\#\mathbb{T}_n^d\leq\sum_{i=1}^k\#\overline{B}(x_i,k-i)=\sum_{j=0}^{k-1}\#\overline{B}(0,j).\]
In particular, we get that
\[b\left(\mathbb{T}_n^d\right)\geq\inf\left\{k\in\mathbb{N}^*:\sum_{j=0}^{k-1}\#\overline{B}(0,j)\geq n^d\right\}=:\kappa_n.\]
This yields the desired lower bound, since $\kappa_n\geq(\alpha_d+o(1))\cdot n^{d/(d+1)}$ as $n\to\infty$.
Indeed, for any sequence $(k_n)_{n\in\mathbb{N}^*}$ of positive integers such that $1\ll k_n\ll n$ as $n\to\infty$, the following holds: for all sufficiently large $n\in\mathbb{N}^*$, we have
\[\#\overline{B}(0,j)=\#\left\{x\in\mathbb{Z}^d:\|x\|_1\leq j\right\}\quad\text{for all $j\in\llbracket0,k_n-1\rrbracket$,}\]
hence
\[\sum_{j=0}^{k_n-1}\#\overline{B}(0,j)\sim\sum_{j=0}^{k_n-1}\frac{2^d}{d!}\cdot j^d\sim\frac{2^d}{d!}\cdot\frac{k_n^{d+1}}{d+1}=\frac{2^d}{(d+1)!}\cdot k_n^{d+1}\quad\text{as $n\to\infty$.}\]
Now, fix a parameter $c\in\left]0,(d+1)!\cdot2^{-d}\right[$, and set
\[k_n(c)=\left\lfloor\left(c\cdot n^d\right)^{1/(d+1)}\right\rfloor\quad\text{for all $n\in\mathbb{N}^*$.}\]
By the previous calculation, we have
\[\sum_{j=0}^{k_n(c)-1}\#\overline{B}(0,j)\sim\frac{2^d}{(d+1)!}\cdot c\cdot n^d\quad\text{as $n\to\infty$.}\]
Thus, for all sufficiently large $n\in\mathbb{N}^*$, we have $\sum_{j=0}^{k_n(c)-1}\#\overline{B}(0,j)<n^d$, hence $\kappa_n>k_n(c)$.
Since $k_n(c)\sim c^{1/(d+1)}\cdot n^{d/(d+1)}$ as $n\to\infty$, it follows that
\[\varliminf_{n\to\infty}\frac{\kappa_n}{n^{d/(d+1)}}\geq c^{1/(d+1)}.\]
Letting $c\to(d+1)!\cdot2^{-d}$, we obtain the desired lower bound:
\[\varliminf_{n\to\infty}\frac{\kappa_n}{n^{d/(d+1)}}\geq\left(\frac{(d+1)!}{2^d}\right)^{1/(d+1)}=\alpha_d.\]

\medskip

\emph{Upper bound.}
Fix a parameter $c\in\mathbb{R}_+^*$ to be adjusted later, and set
\[r_n(c)=\left\lfloor c\cdot n^{d/(d+1)}\right\rfloor\quad\text{for all $n\in\mathbb{N}^*$.}\]
Then, let $x_1,\ldots,x_k\in\mathbb{T}_n^d$ be a maximal sequence of points at distance more than $r_n(c)$ from each other.
Such a sequence can be constructed iteratively as follows.
First, pick an arbitrary point $x_1\in\mathbb{T}_n^d$.
Then, for each $j\in\mathbb{N}^*$ such that $x_1,\ldots,x_j$ have been constructed, proceed as follows:
\begin{itemize}
\item If $\bigcup_{i=1}^j\overline{B}(x_i,r_n(c))=\mathbb{T}_n^d$, then set $k=j$ and terminate.

\item Otherwise, pick an arbitrary point $x_{j+1}\in\left.\mathbb{T}_n^d\middle\backslash\left(\bigcup_{i=1}^j\overline{B}(x_i,r_n(c))\right)\right.$.
\end{itemize}
Note that the procedure terminates in at most $n^d$ steps, since the number of available points drops by at least $1$ at each step.
Moreover, by construction, we have ${\bigcup_{i=1}^k\overline{B}(x_i,r_n(c))=\mathbb{T}_n^d}$.
Furthermore, since the balls $\left(\overline{B}(x_i,\lfloor r_n(c)/2\rfloor)\,;\,i\in\llbracket1,k\rrbracket\right)$ are disjoint, we get that
\[k\leq\frac{n^d}{\#\overline{B}(0,\lfloor r_n(c)/2\rfloor)}.\]
Now, fix arbitrary points ${x_{k+1},\ldots,x_{k+r_n(c)}\in\mathbb{T}_n^d}$, and observe that
\[\bigcup_{i=1}^{k+r_n(c)}\overline{B}(x_i,k+r_n(c)-i)\supset\bigcup_{i=1}^k\overline{B}(x_i,r_n(c))=\mathbb{T}_n^d.\]
By the definition of the burning number, this yields the upper bound
\[b\left(\mathbb{T}_n^d\right)\leq k+r_n(c)\leq\frac{n^d}{\#\overline{B}(0,\lfloor r_n(c)/2\rfloor)}+r_n(c).\]
Since $r_n(c)\sim c\cdot n^{d/(d+1)}\ll n/2$ as $n\to\infty$, we have
\[\frac{n^d}{\#\overline{B}(0,\lfloor r_n(c)/2\rfloor)}=\frac{n^d}{\#\left\{x\in\mathbb{Z}^d:\|x\|_1\leq\lfloor r_n(c)/2\rfloor\right\}}\sim\frac{n^d}{\left.2^d\middle/d!\right.\cdot\lfloor r_n(c)/2\rfloor^d}\sim\frac{d!}{c^d}\cdot n^{d/(d+1)}.\]
Optimising $c$ in order to minimise the constant $d!\cdot c^{-d}+c$, we find that for $c=(d!\cdot d)^{1/(d+1)}$, we have
\[\frac{n^d}{\#\overline{B}(0,\lfloor r_n(c)/2\rfloor)}+r_n(c)\underset{n\to\infty}{\sim}\left(\frac{d!}{(d!\cdot d)^{d/(d+1)}}+(d!\cdot d)^{1/(d+1)}\right)\cdot n^{d/(d+1)}=\gamma_d\cdot n^{d/(d+1)}.\]
This concludes the proof of the upper bound, and of the proposition.
\end{proof}

\section{Coupon collector}\label{sec:couponcollector}

In this section, we consider the coupon collector burning process $\mathcal{B}_n^\mathsf{cc}(\cdot)$, and prove Theorem \ref{thm:couponcollector}.
Recall from the introduction that for each $n\in\mathbb{N}^*$, we let $(X_1,X_2,\ldots)$ be independent random variables with uniform distribution on $\mathbb{T}_n^d$.
Then, we set
\[\mathcal{B}_n^\mathsf{cc}(k)=\mathsf{Burned}(X_1,\ldots,X_k)=\bigcup_{i=1}^k\overline{B}(X_i,k-i)\quad\text{for all $k \in \mathbb{N}$,}\]
and we let
\[\tau_n^\mathsf{cc}=\inf\left\{k\in\mathbb{N}:\mathcal{B}_n^\mathsf{cc}(k)=\mathbb{T}_n^d\right\}.\]
Without further ado, let us dive into the proof of Theorem \ref{thm:couponcollector}.

\begin{proof}[Proof of Theorem \ref{thm:couponcollector}]
The proof is based on a first and second moment argument.

\medskip

\emph{Upper bound.}
Fix a parameter $c\in\mathbb{R}_+^*$, and set (in hindsight)
\[k_n(c)=\left\lfloor\left(c\cdot n^d\cdot\ln\left(n^d\right)\right)^{1/(d+1)}\right\rfloor\quad\text{for all $n\in\mathbb{N}^*$.}\]
We want to show that for each $c>d!\cdot2^{-d}$, we have $\tau_n^\mathsf{cc}\leq k_n (c)$ with high probability as $n\to\infty$.
To this end, let us calculate the first moment of the random variables $\left.\mathbb{T}_n^d\middle\backslash\mathcal{B}_n^\mathsf{cc}(k)\right.$.
We have
\[\mathbb{E}\left[\#\left(\mathbb{T}_n^d\middle\backslash\mathcal{B}_n^\mathsf{cc}(k)\right)\right]=\sum_{x\in\mathbb{T}_n^d}\mathbb{P}\left(x\notin\mathcal{B}_n^\mathsf{cc}(k)\right),\]
where for every $x\in\mathbb{T}_n^d$,
\[\begin{split}
\mathbb{P}\left(x\notin\mathcal{B}_n^\mathsf{cc}(k)\right)&=\mathbb{P}\left(\bigcap_{i=1}^k\left(X_i\notin\overline{B}(x,k-i)\right)\right)\\
&=\prod_{i=1}^k\left(1-\frac{\#\overline{B}(x,k-i)}{n^d}\right)=\prod_{j=0}^{k-1}\left(1-\frac{\#\overline{B}(0,j)}{n^d}\right).
\end{split}\]
Now, we claim that as $n\to\infty$, we have
\begin{eqnarray}
\prod_{j=0}^{k_n(c)-1}\left(1-\frac{\#\overline{B}(0,j)}{n^d}\right)&=&\exp\left(-\sum_{j=0}^{k_n(c)-1}\frac{\#\overline{B}(0,j)}{n^d}+o(1)\right)\label{eq:couponcollector1stmoment1}\\
&=&\exp\left(-\frac{2^d}{(d+1)!}\cdot c\cdot\ln\left(n^d\right)\cdot(1+o(1))\right)\label{eq:couponcollector1stmoment2}
\end{eqnarray}
Taking this for granted, let us conclude the proof of the upper bound.
A naive first moment argument would consist in upper bounding the probability that $\mathbb{T}_n^d$ is not burned by time $k_n(c)$ as follows:
\[\begin{split}
\mathbb{P}\left(\tau_n^\mathsf{cc}>k_n(c)\right)&=\mathbb{P}\left(\left.\mathbb{T}_n^d\middle\backslash\mathcal{B}_n^\mathsf{cc}(k_n(c))\right.\neq\emptyset\right)\\
&\leq\mathbb{E}\left[\#\left(\mathbb{T}_n^d\middle\backslash\mathcal{B}_n^\mathsf{cc}(k_n(c))\right)\right]=n^d\cdot\prod_{j=0}^{k_n(c)-1}\left(1-\frac{\#\overline{B}(0,j)}{n^d}\right).
\end{split}\]
However, by \eqref{eq:couponcollector1stmoment2}, it would require $c>(d+1)!\cdot2^{-d}$ for the right-hand side to go to zero as $n\to\infty$, and this threshold is actually off by a factor of $(d+1)$.
To improve on this first moment argument, let $\left(\overline{B}(x,r_n)\,;\,x\in A_n\right)$ be a covering of $\mathbb{T}_n^d$ by balls of radius $r_n:=\left\lfloor n^{d/(d+1)}\right\rfloor$, with centres $x\in\mathbb{T}_n^d$ more than $r_n$ apart from each other.
Note that such a covering can be constructed iteratively as in the proof of Proposition \ref{prop:detbounds}.
Next, since the balls $\left(\overline{B}(x,\lfloor r_n/2\rfloor)\,;\,x\in A_n\right)$ are disjoint, we have ${\#A_n\leq n^d\cdot\left(\#\overline{B}(0,\lfloor r_n/2\rfloor)\right)^{-1}}$.
In particular, we get that $\#A_n\leq C\cdot n^{d/(d+1)}$ for some constant $C=C(d)\in(0,\infty)$ that does not depend on $n$.
Now, notice that we have the inclusion of events
\[\left(\tau_n^\mathsf{cc}>k_n(c)+r_n\right)=\left(\left.\mathbb{T}_n^d\middle\backslash\mathcal{B}_n^\mathsf{cc}(k_n(c)+r_n)\right.\neq\emptyset\right)\subset\bigcup_{x\in A_n}\left(x\notin\mathcal{B}_n^\mathsf{cc}(k_n(c))\right).\]
Thus, by the union bound, we obtain
\[\begin{split}
\mathbb{P}\left(\tau_n^\mathsf{cc}>k_n(c)+r_n\right)&\leq\sum_{x\in A_n}\mathbb{P}\left(x\notin\mathcal{B}_n^\mathsf{cc}(k_n(c))\right)\\
&=\#A_n\cdot\prod_{j=0}^{k_n(c)-1}\left(1-\frac{\#\overline{B}(0,j)}{n^d}\right)\leq C\cdot n^{d/(d+1)}\cdot\prod_{j=0}^{k_n(c)-1}\left(1-\frac{\#\overline{B}(0,j)}{n^d}\right).
\end{split}\]
Now, by \eqref{eq:couponcollector1stmoment2}, it only takes $c>d!\cdot2^{-d}$ to make the right-hand side go to zero as $n\to\infty$.
This yields the desired upper bound on $\tau_n^\mathsf{cc}$, since $k_n(c)+r_n\sim\left(c\cdot n^d\cdot\ln\left(n^d\right)\right)^{1/(d+1)}$ as $n\to\infty$.

Finally, it remains to prove \eqref{eq:couponcollector1stmoment1} and \eqref{eq:couponcollector1stmoment2}.
To this end, we write
\[\prod_{j=0}^{k_n(c)-1}\left(1-\frac{\#\overline{B}(0,j)}{n^d}\right)=\exp\left(-\sum_{j=0}^{k_n(c)-1}\frac{\#\overline{B}(0,j)}{n^d}+\Delta_n\right),\]
where
\[\Delta_n=\sum_{j=0}^{k_n(c)-1}\left(\frac{\#\overline{B}(0,j)}{n^d}+\ln\left(1-\frac{\#\overline{B}(0,j)}{n^d}\right)\right).\]
Then, since
\[\max_{j\in\llbracket0,k_n(c)-1\rrbracket}\frac{\#\overline{B}(0,j)}{n^d}\leq\frac{\#\overline{B}(0,k_n(c))}{n^d}\underset{n\to\infty}{=}n^{-d/(d+1)+o(1)}\longrightarrow0,\]
we may bound $\Delta_n$ by using the inequality $|u+\ln(1-u)|\leq u^2$ for all $u\in[0,1/2]$.
For all $n\in\mathbb{N}^*$ sufficiently large so that $\#\overline{B}(0,k_n(c))\cdot n^{-d}\leq1/2$, we get
\[|\Delta_n|\leq\sum_{j=0}^{k_n(c)-1}\left(\frac{\#\overline{B}(0,j)}{n^d}\right)^2\leq k_n(c)\cdot\left(\frac{\#\overline{B}(0,k_n(c))}{n^d}\right)^2\underset{n\to\infty}{=}n^{-d/(d+1)+o(1)}\longrightarrow0,\]
which proves \eqref{eq:couponcollector1stmoment1}.
Equation \eqref{eq:couponcollector1stmoment2} now readily follows: we have
\[\sum_{j=0}^{k_n(c)-1}\frac{\#\overline{B}(0,j)}{n^d}\sim\frac{1}{n^d}\cdot\frac{2^d}{d!}\cdot\frac{k_n(c)^{d+1}}{d+1}\sim\frac{2^d}{(d+1)!}\cdot c\cdot\ln\left(n^d\right)\quad\text{as $n\to\infty$.}\]

\medskip

\emph{Lower bound.}
Let us prove that for each $c<d!\cdot2^{-d}$, we have $\tau_n^\mathsf{cc}>k_n(c)$ with high probability as $n\to\infty$.
To this end, let us calculate the second moment of the random variable $\#\left(\mathbb{T}_n^d\middle\backslash\mathcal{B}_n^\mathsf{cc}(k_n(c))\right)$.
Indeed, by the Cauchy--Schwarz inequality, we have
\[\mathbb{P}\left(\tau_n^\mathsf{cc}>k_n(c)\right)=\mathbb{P}\left(\left.\mathbb{T}_n^d\middle\backslash\mathcal{B}_n^\mathsf{cc}(k_n(c))\right.\neq\emptyset\right)\geq\frac{\mathbb{E}\left[\#\left(\mathbb{T}_n^d\middle\backslash\mathcal{B}_n^\mathsf{cc}(k_n(c))\right)\right]^2}{\mathbb{E}\left[\#\left(\mathbb{T}_n^d\middle\backslash\mathcal{B}_n^\mathsf{cc}(k_n(c))\right)^2\right]}.\]
Thus, it suffices to show that for each $c<d!\cdot2^{-d}$, we have
\begin{equation}\label{eq:couponcollector2ndmomentgoal}
\frac{\mathbb{E}\left[\#\left(\mathbb{T}_n^d\middle\backslash\mathcal{B}_n^\mathsf{cc}(k_n(c))\right)^2\right]}{\mathbb{E}\left[\#\left(\mathbb{T}_n^d\middle\backslash\mathcal{B}_n^\mathsf{cc}(k_n(c))\right)\right]^2}\longrightarrow1\quad\text{as $n\to\infty$.}
\end{equation}
The second moment of the random variable $\#\left(\mathbb{T}_n^d\middle\backslash\mathcal{B}_n^\mathsf{cc}(k_n(c))\right)$ can be calculated as follows: we have
\[\mathbb{E}\left[\#\left(\mathbb{T}_n^d\middle\backslash\mathcal{B}_n^\mathsf{cc}(k_n(c))\right)^2\right]=\sum_{x_1,x_2\in\mathbb{T}_n^d}\mathbb{P}\left(x_1,x_2\notin\mathcal{B}_n^\mathsf{cc}(k_n(c))\right),\]
where for every $x_1,x_2\in\mathbb{T}_n^d$,
\[\begin{split}
\mathbb{P}\left(x_1,x_2\notin\mathcal{B}_n^\mathsf{cc}(k_n(c))\right)&=\mathbb{P}\left(\bigcap_{i=1}^{k_n(c)}\left(X_i\notin\left(\overline{B}(x_1,k_n(c)-i)\cup\overline{B}(x_2,k_n(c)-i)\right)\right)\right)\\
&=\prod_{i=1}^{k_n(c)}\left(1-\frac{\#\left(\overline{B}(x_1,k_n(c)-i)\cup\overline{B}(x_2,k_n(c)-i)\right)}{n^d}\right)\\
&=\prod_{j=0}^{k_n(c)-1}\left(1-\frac{\#\left(\overline{B}(0,j)\cup\overline{B}(x_2-x_1,j)\right)}{n^d}\right).
\end{split}\]
Thus, we get
\[\mathbb{E}\left[\#\left(\mathbb{T}_n^d\middle\backslash\mathcal{B}_n^\mathsf{cc}(k_n(c))\right)^2\right]=n^d\cdot\sum_{x\in\mathbb{T}_n^d}\prod_{j=0}^{k_n(c)-1}\left(1-\frac{\#\left(\overline{B}(0,j)\cup\overline{B}(x,j)\right)}{n^d}\right),\]
and it follows that
\begin{eqnarray*}
\lefteqn{\frac{\mathbb{E}\left[\#\left(\mathbb{T}_n^d\middle\backslash\mathcal{B}_n^\mathsf{cc}(k_n(c))\right)^2\right]}{\mathbb{E}\left[\#\left(\mathbb{T}_n^d\middle\backslash\mathcal{B}_n^\mathsf{cc}(k_n(c))\right)\right]^2}}\\
&=&\frac{1}{n^d}\cdot\sum_{x\in\mathbb{T}_n^d}\prod_{j=0}^{k_n(c)-1}\left(\left(1-\frac{\#\left(\overline{B}(0,j)\cup\overline{B}(x,j)\right)}{n^d}\right)\cdot\left(1-\frac{\#\overline{B}(0,j)}{n^d}\right)^{-2}\right).
\end{eqnarray*}
Now, we claim that as $n\to\infty$, we have
\begin{equation}\label{eq:couponcollector2ndmoment}
\prod_{j=0}^{k_n(c)-1}\left(1-\frac{\#\left(\overline{B}(0,j)\cup\overline{B}(x,j)\right)}{n^d}\right)=\exp\left(-\sum_{j=0}^{k_n(c)-1}\frac{\#\left(\overline{B}(0,j)\cup\overline{B}(x,j)\right)}{n^d}+o(1)\right)
\end{equation}
uniformly in $x\in\mathbb{T}_n^d$.
Indeed, we can write
\[\prod_{j=0}^{k_n(c)-1}\left(1-\frac{\#\left(\overline{B}(0,j)\cup\overline{B}(x,j)\right)}{n^d}\right)=\exp\left(-\sum_{j=0}^{k_n(c)-1}\frac{\#\left(\overline{B}(0,j)\cup\overline{B}(x,j)\right)}{n^d}+\Delta_n(x)\right),\]
where
\[\Delta_n(x)=\sum_{j=0}^{k_n(c)-1}\left(\frac{\#\left(\overline{B}(0,j)\cup\overline{B}(x,j)\right)}{n^d}+\ln\left(1-\frac{\#\left(\overline{B}(0,j)\cup\overline{B}(x,j)\right)}{n^d}\right)\right).\]
Then, since
\[\max_{j\in\llbracket0,k_n(c)-1\rrbracket}\frac{\#\left(\overline{B}(0,j)\cup\overline{B}(x,j)\right)}{n^d}\leq\frac{2\cdot\#\overline{B}(0,k_n(c))}{n^d}\underset{n\to\infty}{=}n^{-d/(d+1)+o(1)}\longrightarrow0,\]
we may bound $\Delta_n(x)$ by using the inequality $|u+\ln(1-u)|\leq u^2$ for all $u\in[0,1/2]$.
For all $n\in\mathbb{N}^*$ sufficiently large so that $\left(2\cdot\#\overline{B}(0,k_n(c))\right)\cdot n^{-d}\leq1/2$, we get
\[|\Delta_n(x)|\leq\sum_{j=0}^{k_n(c)-1}\left(\frac{\#\left(\overline{B}(0,j)\cup\overline{B}(x,j)\right)}{n^d}\right)^2\leq k_n(c)\cdot\left(\frac{2\cdot\#\overline{B}(0,k_n(c))}{n^d}\right)^2.\]
This proves \eqref{eq:couponcollector2ndmoment}, since the right-hand side does not depend on $x$ and goes to zero as $n\to\infty$.

Now, combining \eqref{eq:couponcollector1stmoment1} and \eqref{eq:couponcollector2ndmoment}, we get that as $n\to\infty$, we have
\begin{eqnarray*}
\lefteqn{\prod_{j=0}^{k_n(c)-1}\left(\left(1-\frac{\#\left(\overline{B}(0,j)\cup\overline{B}(x,j)\right)}{n^d}\right)\cdot\left(1-\frac{\#\overline{B}(0,j)}{n^d}\right)^{-2}\right)}\\
&=&\exp\left(\sum_{j=0}^{k_n(c)-1}\frac{\#\left(\overline{B}(0,j)\cap\overline{B}(x,j)\right)}{n^d}+o(1)\right)\quad\text{uniformly in $x\in\mathbb{T}_n^d$.}
\end{eqnarray*}
We deduce that
\begin{equation}\label{eq:couponcollector2ndmomenteq}
\frac{\mathbb{E}\left[\#\left(\mathbb{T}_n^d\middle\backslash\mathcal{B}_n^\mathsf{cc}(k_n(c))\right)^2\right]}{\mathbb{E}\left[\#\left(\mathbb{T}_n^d\middle\backslash\mathcal{B}_n^\mathsf{cc}(k_n(c))\right)\right]^2}\underset{n\to\infty}{\sim}\frac{1}{n^d}\cdot\sum_{x\in\mathbb{T}_n^d}\exp\left(\sum_{j=0}^{k_n(c)-1}\frac{\#\left(\overline{B}(0,j)\cap\overline{B}(x,j)\right)}{n^d}\right).
\end{equation}
Finally, we have everything at hand to prove that for each $c<d!\cdot2^{-d}$, the convergence \eqref{eq:couponcollector2ndmomentgoal} holds.
We upper bound the right-hand side of \eqref{eq:couponcollector2ndmomenteq} as follows: since $\overline{B}(0,j)\cap\overline{B}(x,j)=\emptyset$ as soon as $d(0,x)>2k_n(c)$, we have
\[\frac{1}{n^d}\cdot\sum_{x\in\mathbb{T}_n^d}\exp\left(\sum_{j=0}^{k_n(c)-1}\frac{\#\left(\overline{B}(0,j)\cap\overline{B}(x,j)\right)}{n^d}\right)\leq1+\frac{1}{n^d}\cdot\#\overline{B}(0,2k_n(c))\cdot\exp\left(\sum_{j=0}^{k_n(c)-1}\frac{\#\overline{B}(0,j)}{n^d}\right).\]
Now, observe that on the one hand, we have $\#\overline{B}(0,2k_n(c))\cdot n^{-d}=n^{-d/(d+1)+o(1)}$ as $n\to\infty$, and on the other hand, we have
\[\sum_{j=0}^{k_n(c)-1}\frac{\#\overline{B}(0,j)}{n^d}\sim\frac{1}{n^d}\cdot\frac{2^d}{d!}\cdot\frac{k_n(c)^{d+1}}{d+1}\sim\frac{2^d}{d!}\cdot c\cdot\ln\left(n^{d/(d+1)}\right)\quad\text{as $n\to\infty$,}\]
with $\left.2^d\middle/d!\right.\cdot c<1$ by assumption.
Therefore, we obtain that
\[\frac{1}{n^d}\cdot\#\overline{B}(0,2k_n(c))\cdot\exp\left(\sum_{j=0}^{k_n(c)-1}\frac{\#\overline{B}(0,j)}{n^d}\right)\longrightarrow0\quad\text{as $n\to\infty$,}\]
which yields \eqref{eq:couponcollector2ndmomentgoal}.
This completes the proof of the lower bound, and of the theorem.
\end{proof}

\section{Rejection sampling}\label{sec:rejectionsampling}

In this section, we consider the rejection sampling burning process $\mathcal{B}_n^\mathsf{rs}(\cdot)$, and prove Theorem~\ref{thm:rejectionsampling}.
Recall from the introduction that for each $n\in\mathbb{N}^*$, we consider a sequence $(Y_1,Y_2,\ldots)$ of random variables with values in $\mathbb{T}_n^d$ which satisfies the following condition (RS): for each $k\in\mathbb{N}$, conditionally on $Y_1,\ldots,Y_k$, the random variable $Y_{k+1}$ has uniform distribution on the complement of $\mathsf{Burned}(Y_1,\ldots,Y_k)$, provided that $\mathsf{Burned}(Y_1,\ldots,Y_k)\varsubsetneq\mathbb{T}_n^d$.
Then, we set
\[\mathcal{B}_n^\mathsf{rs}(k)=\mathsf{Burned}(Y_1,\ldots,Y_k)=\bigcup_{i=1}^k\overline{B}(Y_i,k-i)\quad\text{for all $k\in\mathbb{N}$,}\]
and we let
\[\tau_n^\mathsf{rs}=\inf\left\{k\in\mathbb{N}:\mathcal{B}_n^\mathsf{rs}(k)=\mathbb{T}_n^d\right\}.\]
Note that even though the distribution of the sequence $(Y_1,Y_2,\ldots)$ is not uniquely determined by condition (RS), since it does not prescribe the conditional distribution of $Y_{k+1}$ given $Y_1,\ldots,Y_k$ on the event $\left(\mathsf{Burned}(Y_1,\ldots,Y_k)=\mathbb{T}_n^d\right)$, this affects neither the distribution of the random burning number $\tau_n^\mathsf{rs}$ nor that of the burning process $\mathcal{B}_n^\mathsf{rs}(\cdot)$, for as soon as ${\mathsf{Burned}(Y_1,\ldots,Y_k)=\mathbb{T}_n^d}$, we have $\mathsf{Burned}(Y_1,\ldots,Y_k,Y_{k+1})=\mathbb{T}_n^d$ regardless of $Y_{k+1}$.
In Subsection~\ref{subsec:coupling}, we construct a sequence satisfying condition (RS) out of a collection of independent random variables with uniform distribution on $\mathbb{T}_n^d$, by using rejection sampling.
Now, recall the statement of Theorem \ref{thm:rejectionsampling}, which consists of \eqref{eq:asymptoticrandomburning} and \eqref{eq:functionalconvergence}, and involves the maximal solution ${y:{[0,T[}\rightarrow\mathbb{R}}$ of the Cauchy problem \eqref{eq:blasius}.
This solution has been extensively studied in \cite{blasius}.
In particular, it has been shown that this solution blows up in finite time, i.e, we have $T\in(0,\infty)$.
Moreover, for each $i\in\llbracket0,d\rrbracket$, the function $y^{(i)}$ is strictly increasing on $[0,T[$, and we have $y^{(i)}(t)\rightarrow\infty$ as $t\to T^-$.
Furthermore, for each $i\in\llbracket0,d\rrbracket$, we have
\begin{equation}\label{eq:blowuprate}
y^{(i)}(t)\asymp\frac{1}{(T-t)^{i+1}}\quad\text{as $t\to T^-$.}
\end{equation}
This last equation will be particularly useful to us for $i=d$.
Regarding the proof of Theorem~\ref{thm:rejectionsampling}, we first prove \eqref{eq:functionalconvergence} in Subsection \ref{subsec:functionalconvergence}, and then prove \eqref{eq:asymptoticrandomburning} in Subsection \ref{subsec:asymptoticrandomburning} (indeed, Equation \eqref{eq:asymptoticrandomburning} is not a direct consequence of \eqref{eq:functionalconvergence}).

\subsection{Coupling arguments}\label{subsec:coupling}

For each $n\in\mathbb{N}^*$, let $\left(X_i^h\,;\,i,h\in\mathbb{N}^*\right)$ be a collection of independent random variables with uniform distribution on $\mathbb{T}_n^d$.
In this subsection, we construct an instance of the rejection sampling burning process $\mathcal{B}_n^\mathsf{rs}(\cdot)$ out of the random variables $\left(X_i^h\,;\,i,h\in\mathbb{N}^*\right)$, by using rejection sampling.
Then, we use these same random variables to couple the rejection sampling burning process with intermediate burning processes, which will be instrumental in the proof of Theorem \ref{thm:rejectionsampling}.
For each $k\in\mathbb{N}$, we denote by $\mathcal{F}_k$ the $\sigma$-algebra generated by the random variables
\[\left(X_i^h\,;\,\text{$i\in\llbracket1,k\rrbracket$ and $h\in\mathbb{N}^*$}\right).\]

\paragraph{An instance of the rejection sampling burning process.}

Let $\left(Y_1^*,Y_2^*,\ldots\right)$ be the sequence of random variables with values in $\mathbb{T}_n^d$ constructed out of the random variables $\left(X_i^h\,;\,i,h\in\mathbb{N}^*\right)$ by induction as follows: for each $k\in\mathbb{N}$ such that $Y_1^*,\ldots,Y_k^*$ have been constructed, we set
\[H_{k+1}^*=\inf\left\{h\in\mathbb{N}^*:X_{k+1}^h\notin\mathsf{Burned}\left(Y_1^*,\ldots,Y_k^*\right)\right\}\in\llbracket1,\infty\rrbracket,\]
and we let
\[Y_{k+1}^*=\begin{cases}
X_{k+1}^{H_{k+1}^*}&\text{if $H_{k+1}^*<\infty$,}\\
X_{k+1}^1&\text{otherwise.}
\end{cases}\]
We then denote by $\mathcal{B}_n^*(\cdot)$ the burning process associated with the sequence $\left(Y_1^*,Y_2^*,\ldots\right)$: we have
\[\mathcal{B}_n^*(k)=\mathsf{Burned}(Y_1^*,\ldots,Y_k^*)=\bigcup_{i=1}^k\overline{B}\left(Y_i^*,k-i\right)\quad\text{for all $k\in\mathbb{N}$.}\]
In words, what we do at step $(k+1)$ in the burning process $\mathcal{B}_n^*(\cdot)$ is sample independent uniforms, namely $X_{k+1}^1,X_{k+1}^2,\ldots$, until we find one that does not belong to the set of vertices burned at step $k$, namely $\mathsf{Burned}(Y_1^*,\ldots,Y_k^*)$.
Then, we set the corresponding point $X_{k+1}^{H_{k+1}^*}$ on fire, coming up with the arbitrary convention of setting $X_{k+1}^1$ on fire if we do not find such a uniform, i.e, if $H_{k+1}^*=\infty$.
In particular, note that $H_1^*=1$ and $Y_1^*=X_1^1$.
Now, by construction, the burning process $\mathcal{B}_n^*(\cdot)$ is an instance of the rejection sampling burning process $\mathcal{B}_n^\mathsf{rs}(\cdot)$, as stated in the following straightforward claim:

\begin{claim}\label{claim:rejectionsampling}
For each $k\in\mathbb{N}$, the random variables $Y_1^*,\ldots,Y_k^*$ are measurable with respect to $\mathcal{F}_k$, and the law of $Y_{k+1}^*$ conditionally on $\mathcal{F}_k$ is the uniform distribution on the complement of $\mathcal{B}_n^*(k)$, provided that ${\mathcal{B}_n^*(k)\varsubsetneq\mathbb{T}_n^d}$.
\end{claim}

Unfortunately, this burning process $\mathcal{B}_n^*(\cdot)$ is hard to analyse directly, since the choice of the new point that we set on fire at each step depends on the state of the process $\mathcal{B}_n^*(\cdot)$ itself at the previous step.
However, by construction, the burning process $\mathcal{B}_n^*(\cdot)$ contains an instance of the coupon collector burning process $\mathcal{B}_n^\mathsf{cc}(\cdot)$: indeed, for all $k\in\mathbb{N}$, we have
\[\mathcal{B}_n^*(k)\supset\bigcup_{i=1}^k\overline{B}\left(X_i^1,k-i\right).\]
See Claim \ref{claim:inclusionintermediate} below for a formal proof.
The modest observations above are at the basis of our idea to introduce the following intermediate burning processes $\left(\mathcal{B}_n^p(\cdot)\,;\,p\in\mathbb{N}^*\right)$, which will interpolate between the coupon collector and the rejection sampling burning processes.

\paragraph{The intermediate burning processes $\left(\mathcal{B}_n^p(\cdot)\,;\,p\in\mathbb{N}^*\right)$.}
Let $\left(\mathcal{B}_n^p(\cdot)\,;\,p\in\mathbb{N}\right)$ be the sequences of random subsets of $\mathbb{T}_n^d$ constructed out of the random variables $\left(X_i^h\,;\,i,h\in\mathbb{N}^*\right)$ by induction as follows.
For ${p=0}$, we set $\mathcal{B}_n^0(k)=\emptyset$ for all $k\in\mathbb{N}$.
Then, for each $p\in\mathbb{N}$ such that $\mathcal{B}_n^0(\cdot),\ldots,\mathcal{B}_n^p(\cdot)$ have been constructed, we construct $\mathcal{B}_n^{p+1}(\cdot)$ as follows.
For all $k\in\mathbb{N}$, we set
\[H_{k+1}^{p+1}=\inf\left\{h\in\mathbb{N}^*:X_{k+1}^h\notin\mathcal{B}_n^p(k)\right\}\in\llbracket1,\infty\rrbracket,\]
and we let
\[Y_{k+1}^{p+1}=\begin{cases}
X_{k+1}^{H_{k+1}^{p+1}}&\text{if $H_{k+1}^{p+1}<\infty$,}\\
X_{k+1}^1&\text{otherwise.}
\end{cases}\]
We then define $\mathcal{B}_n^{p+1}(\cdot)$ as the burning process associated with the sequence $\left(Y_1^{p+1},Y_2^{p+1},\ldots\right)$: we have\footnote{In particular, note that $\mathcal{B}_n^{p+1}(0)=\emptyset$.}
\[\mathcal{B}_n^{p+1}(k)=\mathsf{Burned}\left(Y_1^{p+1},\ldots,Y_k^{p+1}\right)=\bigcup_{i=1}^k\overline{B}\left(Y_i^{p+1},k-i\right)\quad\text{for all $k\in\mathbb{N}$.}\]
In words, what we do at step $(k+1)$ in the burning process $\mathcal{B}_n^{p+1}(\cdot)$ is sample independent uniforms, namely $X_{k+1}^1,X_{k+1}^2,\ldots$ (crucially we use the same uniforms as for the burning process $\mathcal{B}_n^*(\cdot)$ and between different values of $p$), until we find one that does not belong to the set of vertices burned at step $k$ in the previous process, namely $\mathcal{B}_n^p(k)$.
Then, we set the corresponding point $X_{k+1}^{H_{k+1}^{p+1}}$ on fire, coming up with the arbitrary convention of setting $X_{k+1}^1$ on fire if we do not find such a uniform, i.e, if $H_{k+1}^{p+1}=\infty$.
Note that $H^1_k=1$ for all $k\in\mathbb{N}^*$, hence $Y^1_k=X_1^1$ for all $k\in\mathbb{N}^*$.
In particular, the burning process $\mathcal{B}_n^1(\cdot)$ is an instance of the coupon collector burning process $\mathcal{B}_n^\mathsf{cc}(\cdot)$.
Now, by construction, the following straightforward analogue of Claim \ref{claim:rejectionsampling} holds:

\begin{claim}\label{claim:rejectionsamplingp}
For each $p\in\mathbb{N}$, the following holds for all $k\in\mathbb{N}$:
\begin{itemize}
\item The random subset $\mathcal{B}_n^p(k)$ and the random variables $Y_1^{p+1},\ldots,Y_k^{p+1}$ are measurable with respect to $\mathcal{F}_k$,

\item The law of $Y_{k+1}^{p+1}$ conditionally on $\mathcal{F}_k$ is the uniform distribution on the complement of $\mathcal{B}_n^p(k)$, provided that ${\mathcal{B}_n^p(k)\varsubsetneq\mathbb{T}_n^d}$.
\end{itemize}
\end{claim}

Moreover, in this coupling, we have the following inclusions, which formalise the fact that the intermediate burning processes $\left(\mathcal{B}_n^p(\cdot)\,;\,p\in\mathbb{N}^*\right)$ interpolate between the coupon collector and the rejection sampling burning processes:

\begin{claim}\label{claim:inclusionintermediate}
For each $p\in\mathbb{N}$, we have $\mathcal{B}_n^p(k)\subset\mathcal{B}_n^{p+1}(k)\subset\mathcal{B}_n^*(k)$ for all $k\in\mathbb{N}$.
In particular, we have $H_i^{p+1}\leq H_i^{p+2}\leq H_i^*$ for all $i\in\mathbb{N}^*$.
\end{claim}

Let us give a detailed proof of this fact.

\begin{proof}
To begin with, let us prove by induction on $p\in\mathbb{N}$ that for all $k\in\mathbb{N}$, we have
\begin{equation}\label{eq:inclusion}
\mathcal{B}_n^p(k)\subset\mathcal{B}_n^{p+1}(k)=\bigcup_{i=1}^k\bigcup_{h=1}^{H_i^{p+1}}\overline{B}\left(X_i^h,k-i\right).
\end{equation}
Equation \eqref{eq:inclusion} holds for $p=0$: indeed, we have $\mathcal{B}_n^p(k)=\emptyset$ for all $k\in\mathbb{N}$, and thus $H_i^1=1$ for all $i\in\mathbb{N}^*$.
Next, fix $p\in\mathbb{N}$ such that \eqref{eq:inclusion} holds, and let us prove that for all $k\in\mathbb{N}$, we have
\begin{equation}\label{eq:inclusion'''}
\mathcal{B}_n^{p+1}(k)\subset\mathcal{B}_n^{p+2}(k)=\bigcup_{i=1}^k\bigcup_{h=1}^{H_i^{p+2}}\overline{B}\left(X_i^h,k-i\right).
\end{equation}
First, we prove by induction on $k\in\mathbb{N}$ that
\begin{equation}\label{eq:inclusion'}
\mathcal{B}_n^{p+2}(k)=\bigcup_{i=1}^k\bigcup_{h=1}^{H_i^{p+2}}\overline{B}\left(X_i^h,k-i\right).
\end{equation}
This equality is vacuous for $k=0$.
Then, fix $k\in\mathbb{N}$ such that \eqref{eq:inclusion'} holds, and let us prove the inclusion
\begin{equation}\label{eq:inclusion''}
\bigcup_{i=1}^{k+1}\bigcup_{h=1}^{H_i^{p+2}}\overline{B}\left(X_i^h,k+1-i\right)\subset\mathcal{B}_n^{p+2}(k+1),
\end{equation}
the converse inclusion being automatic.
On the one hand, by the induction hypothesis \eqref{eq:inclusion'} and since $\mathcal{B}_n^{p+2}(\cdot)$ is a burning process, we have
\[\bigcup_{i=1}^k\bigcup_{h=1}^{H_i^{p+2}}\overline{B}\left(X_i^h,k+1-i\right)\subset\bigcup_{w\in\mathcal{B}_n^{p+2}(k)}\overline{B}(w,1)\subset\mathcal{B}_n^{p+2}(k+1).\]
On the other hand, we also have
\[\bigcup_{h=1}^{H_{k+1}^{p+2}}\left\{X_{k+1}^h\right\}\subset\mathcal{B}_n^{p+2}(k+1),\]
since for each $h\in\mathbb{N}^*$ such that $h<H_{k+1}^{p+2}$, we have $X_{k+1}^h\in\mathcal{B}_n^{p+1}(k)\subset\mathcal{B}_n^{p+2}(k)$.
Altogether, we obtain the desired inclusion \eqref{eq:inclusion''}, which completes the proof of \eqref{eq:inclusion'} by induction.
Now, to complete the proof of \eqref{eq:inclusion'''}, it remains to show that $\mathcal{B}_n^{p+1}(k)\subset\mathcal{B}_n^{p+2}(k)$ for all $k\in\mathbb{N}$.
By the induction hypothesis \eqref{eq:inclusion}, we have $\mathcal{B}_n^p(k)\subset\mathcal{B}_n^{p+1}(k)$ for all $k\in\mathbb{N}$, which entails that $H_i^{p+1}\leq H_i^{p+2}$ for all $i\in\mathbb{N}^*$.
This readily yields the desired inclusion: for all $k\in\mathbb{N}$, we have
\[\mathcal{B}_n^{p+1}(k)=\bigcup_{i=1}^k\bigcup_{h=1}^{H_i^{p+1}}\overline{B}\left(X_i^h,k-i\right)\subset\bigcup_{i=1}^k\bigcup_{h=1}^{H_i^{p+2}}\overline{B}\left(X_i^h,k-i\right)=\mathcal{B}_n^{p+2}(k).\]
Finally, to complete the proof of the claim, it remains to show that for each $p\in\mathbb{N}$, we have $\mathcal{B}_n^p(k)\subset\mathcal{B}_n^*(k)$ for all $k\in\mathbb{N}$.
First, similarly as for \eqref{eq:inclusion'}, one can prove by induction on $k\in\mathbb{N}$ that
\[\mathcal{B}_n^*(k)=\bigcup_{i=1}^k\bigcup_{h=1}^{H_i^*}\overline{B}\left(X_i^h,k-i\right).\]
Let us avoid repeating these details.
From there, it is not difficult to prove the desired inclusion, namely that $\mathcal{B}_n^p(k)\subset\mathcal{B}_n^*(k)$ for all $k\in\mathbb{N}$, by induction on $p\in\mathbb{N}$.
This inclusion is vacuous for $p=0$, since $\mathcal{B}_n^0(k)=\emptyset$ for all $k\in\mathbb{N}$.
Next, fix $p\in\mathbb{N}$ such that $\mathcal{B}_n^p(k)\subset\mathcal{B}_n^*(k)$ for all $k\in\mathbb{N}$, and let us prove that $\mathcal{B}_n^{p+1}(k)\subset\mathcal{B}_n^*(k)$ for all $k\in\mathbb{N}$.
By the induction hypothesis, we have $H_i^{p+1}\leq H_i^*$ for all $i\in\mathbb{N}^*$, which readily yields the desired inclusion: for all $k\in\mathbb{N}$, we have
\[\mathcal{B}_n^{p+1}(k)=\bigcup_{i=1}^k\bigcup_{h=1}^{H_i^{p+1}}\overline{B}\left(X_i^h,k-i\right)\subset\bigcup_{i=1}^k\bigcup_{h=1}^{H_i^*}\overline{B}\left(X_i^h,k-i\right)=\mathcal{B}_n^*(k).\]
The proof of the claim is now complete.
\end{proof}

\subsection{Functional convergence}\label{subsec:functionalconvergence}

In this subsection, we prove the functional convergence \eqref{eq:functionalconvergence}.
We rely on the coupling of the rejection sampling burning process with the intermediate burning processes $\left(\mathcal{B}_n^p(\cdot)\,;\,p\in\mathbb{N}^*\right)$, introduced in the previous subsection.
The following proposition identifies the scaling limit of the processes $\#\mathcal{B}_n^p(\cdot)$ (which also turns out to be more convenient to state for the complements).
It involves the sequence of functions $\left(f_p:\mathbb{R}_+\rightarrow{[1,\infty[}\right)_{p\in\mathbb{N}}$ defined by induction as follows: set $f_0(t)=1$ for all $t\in\mathbb{R}_+$, and for each $p\in\mathbb{N}$, let
\[f_{p+1}(t)=\exp\left(\frac{2^d}{d!}\cdot\int_0^t(t-s)^d\cdot f_p(s)\mathrm{d}s\right)\quad\text{for all $t\in\mathbb{R}_+$.}\]

\begin{prop}\label{prop:limitingdensityBnp}
For each $p\in\mathbb{N}$, the following holds: for every $\tau\in(0,\infty)$, we have
\begin{equation}\label{eq:functionalconvergencep}
\sup_{t\in[0,\tau]}\left|\frac{\#\left(\mathbb{T}_n^d\middle\backslash\mathcal{B}_n^p\left(\left\lfloor t\cdot n^{d/(d+1)}\right\rfloor\right)\right)}{n^d}-\frac{1}{f_p(t)}\right|\overset{\mathbb{P}}{\longrightarrow}0\quad\text{as $n\to\infty$.}
\end{equation}
\end{prop}

Before getting into the proof of Proposition \ref{prop:limitingdensityBnp}, in view of \eqref{eq:functionalconvergence}, let us highlight the link between the sequence $(f_p:\mathbb{R}_+\rightarrow[1,\infty[)_{p\in\mathbb{N}}$ and the function $y^{(d)}:{[0,T[}\rightarrow\mathbb{R}$.
Recall that $y:{[0,T[}\rightarrow\mathbb{R}$ is the maximal solution of the generalised Blasius equation \eqref{eq:blasius}.

\begin{claim}\label{claim:picard}
We have
\begin{equation}\label{eq:yd}
y^{(d)}(t)=\exp\left(\frac{2^d}{d!}\cdot\int_0^t(t-s)^d\cdot y^{(d)}(s)\mathrm{d}s\right)\quad\text{for all $t\in[0,T[$.}
\end{equation}
Moreover, for each $\tau\in[0,T[$, we have $\sup_{t\in[0,\tau]}\left|f_p(t)-y^{(d)}(t)\right|\rightarrow0$ as $p\to\infty$.
\end{claim}
\begin{proof}
We start with the proof of \eqref{eq:yd}: by integrating \eqref{eq:blasius} as a first-order differential equation in $y^{(d)}$, we get
\[y^{(d)}(t)=\exp\left(2^d\cdot\int_0^ty(s)\mathrm{d}s\right)\quad\text{for all $t\in[0,T[$,}\]
which yields \eqref{eq:yd} by successive integration by parts.
Now, we turn to the proof of the second point.
By Dini's theorem, it suffices to prove that for each $t\in[0,T[$, we have $f_p(t)\uparrow y^{(d)}(t)$ as $p\uparrow\infty$.
To begin with, it is easy to check by induction on $p\in\mathbb{N}$ that $f_{p+1}(t)\geq f_p(t)$ for all $t\in\mathbb{R}_+$.
Next, let us check, again by induction on $p\in\mathbb{N}$, that $f_p(t)\leq y^{(d)}(t)$ for all $t\in[0,T[$.
First, from \eqref{eq:yd}, we infer that $y^{(d)}(t)\geq0$ for all $t\in[0,T[$, which in turn implies that $y^{(d)}(t)\geq1=f_0(t)$ for all $t\in[0,T[$.
Next, fix $p\in\mathbb{N}$ such that $f_p(t)\leq y^{(d)}(t)$ for all $t\in[0,T[$, and let us prove that $f_{p+1}(t)\leq y^{(d)}(t)$ for all $t\in[0,T[$.
By \eqref{eq:yd}, we get that for all $t\in[0,T[$, we have
\[f_{p+1}(t)=\exp\left(\frac{2^d}{d!}\cdot\int_0^t(t-s)^d\cdot f_p(s)\mathrm{d}s\right)\leq\exp\left(\frac{2^d}{d!}\cdot\int_0^t(t-s)^d\cdot y^{(d)}(s)\mathrm{d}s\right)=y^{(d)}(t),\]
which concludes the proof by induction.
Now, by monotonicity and boundedness, we deduce that the limit ${f(t):=\lim_{p\uparrow\infty}f_p(t)}$ exists for all $t\in[0,T[$.
Moreover, by the monotone convergence theorem, we obtain
\[f(t)=\exp\left(\frac{2^d}{d!}\cdot\int_0^t(t-s)^d\cdot f(s)\mathrm{d}s\right)\quad\text{for all $t\in[0,T[$.}\]
This identity entails that $f(t)=y^{(d)}(t)$ for all $t\in[0,T[$.
To prove this, set $F_0(t)=f(t)$ for all $t\in[0,T[$, and by induction, for each $i\in\mathbb{N}$ such that $F_0,\ldots,F_i$ have been constructed, let $F_{i+1}(t)=\int_0^tF_i(s)\mathrm{d}s$ for all $t\in[0,T[$.
By successive integration by parts, we get
\[f(t)=\exp\left(2^d\cdot\int_0^tF_d(s)\mathrm{d}s\right)\quad\text{for all $t\in[0,T[$.}\]
In particular, the function $f$ is differentiable, and we have
\[f'(t)=2^d\cdot F_d(t)\cdot\exp\left(2^d\cdot\int_0^tF_d(s)\mathrm{d}s\right)\quad\text{for all $t\in[0,T[$.}\]
In other words, the function $F_d$ is $(d+1)$ times differentiable, and we have
\[F_d^{(d+1)}(t)=2^d\cdot F_d(t)\cdot F_d^{(d)}(t)\quad\text{for all $t\in[0,T[$.}\]
Since moreover we have $F_d(0)=\ldots=F_d^{(d-1)}(0)=0$ and $F_d^{(d)}(0)=f(0)=1$, we conclude that $F_d(t)=y(t)$ for all $t\in[0,T[$, which finally yields that $f(t)=F_d^{(d)}(t)=y^{(d)}(t)$ for all $t\in[0,T[$, as desired.
\end{proof}

Now, let us come to the proof of Proposition \ref{prop:limitingdensityBnp}.

\begin{proof}[Proof of Proposition \ref{prop:limitingdensityBnp}]
We proceed by induction on $p\in\mathbb{N}$.
First, the functional convergence \eqref{eq:functionalconvergencep} holds vacuously for $p=0$.
Next, let $p\in\mathbb{N}$ be such that \eqref{eq:functionalconvergencep} holds, and let us show that for each $\tau\in\mathbb{R}_+^*$, we have
\[\sup_{t\in[0,\tau]}\left|\frac{\#\left(\mathbb{T}_n^d\middle\backslash\mathcal{B}_n^{p+1}\left(\left\lfloor t\cdot n^{d/(d+1)}\right\rfloor\right)\right)}{n^d}-\frac{1}{f_{p+1}(t)}\right|\overset{\mathbb{P}}{\longrightarrow}0\quad\text{as $n\to\infty$.}\]
By a probabilistic version of Dini's theorem (see, e.g, \cite[Lemma 11.6]{chef}), it suffices to show that for each $t\in\mathbb{R}_+$, we have
\begin{equation}\label{eq:convoitedconv}
\frac{\#\left(\mathbb{T}_n^d\middle\backslash\mathcal{B}_n^{p+1}\left(\left\lfloor t\cdot n^{d/(d+1)}\right\rfloor\right)\right)}{n^d}\longrightarrow\frac{1}{f_{p+1}(t)}\quad\text{in probability as $n\to\infty$.}
\end{equation}
The proof of this convergence is based on a first and second moment argument: fix $t\in\mathbb{R}_+$, set $k_n=\left\lfloor t\cdot n^{d/(d+1)}\right\rfloor$ for short, and let us estimate the first and second moments of the random variables $\#\left(\mathbb{T}_n^d\middle\backslash\mathcal{B}_n^{p+1}(k_n)\right)$ as $n\to\infty$.
To this end, let us estimate the probabilities
\[\left(\mathbb{P}\left(x_1,x_2\notin\mathcal{B}_n^{p+1}(k_n)\right)\,;\,x_1,x_2\in\mathbb{T}_n^d\right).\]
For each $x_1,x_2\in\mathbb{T}_n^d$, we set
\[B_i^{x_1,x_2}=\overline{B}(x_1,k_n-i)\cup\overline{B}(x_2,k_n-i)\quad\text{for all $i\in\llbracket1,k_n\rrbracket$.}\]
We claim that as $n\to\infty$, we have
\begin{equation}\label{eq:keylemestimate}
\exp\left(\sum_{i=1}^{k_n}\frac{\#B_i^{x_1,x_2}}{n^d}\cdot f_p\left(\frac{i-1}{n^{d/(d+1)}}\right)\right)\cdot\mathbb{P}\left(x_1,x_2\notin\mathcal{B}_n^{p+1}(k_n)\right)\longrightarrow1\quad\text{uniformly in $x_1,x_2\in\mathbb{T}_n^d$.}
\end{equation}
Taking this for granted, let us conclude our first and second moment argument.
First, by \eqref{eq:keylemestimate}, we obtain that
\[\exp\left(\sum_{i=1}^{k_n}\frac{\#B_i^{x,x}}{n^d}\cdot f_p\left(\frac{i-1}{n^{d/(d+1)}}\right)\right)\cdot\mathbb{P}\left(x\notin\mathcal{B}_n^{p+1}(k_n)\right)\underset{n\to\infty}{\longrightarrow}1\quad\text{uniformly in $x\in\mathbb{T}_n^d$.}\]
Now, notice that
\[\sum_{i=1}^{k_n}\frac{\#B_i^{x,x}}{n^d}\cdot f_p\left(\frac{i-1}{n^{d/(d+1)}}\right)=\sum_{i=1}^{k_n}\frac{\#\overline{B}(0,k_n-i)}{n^d}\cdot f_p\left(\frac{i-1}{n^{d/(d+1)}}\right)\]
does not depend on $x$.
Moreover, observe that
\begin{eqnarray}
\lefteqn{\sum_{i=1}^{k_n}\frac{\#\overline{B}(0,k_n-i)}{n^d}\cdot f_p\left(\frac{i-1}{n^{d/(d+1)}}\right)}\notag\\
&=&\int_0^{k_n}\frac{\#\overline{B}(0,k_n-\lceil u\rceil)}{n^d}\cdot f_p\left(\frac{\lfloor u\rfloor}{n^{d/(d+1)}}\right)\mathrm{d}u\notag\\
&=&\int_0^{k_n\cdot n^{-d/(d+1)}}\frac{\#\overline{B}\left(0,k_n-\left\lceil s\cdot n^{d/(d+1)}\right\rceil\right)}{\left(n^{d/(d+1)}\right)^d}\cdot f_p\left(\frac{\left\lfloor s\cdot n^{d/(d+1)}\right\rfloor}{n^{d/(d+1)}}\right)\mathrm{d}s\notag\\
&\underset{n\to\infty}{\longrightarrow}&\frac{2^d}{d!}\cdot\int_0^t(t-s)^d\cdot f_p(s)\mathrm{d}s\label{eq:untag},
\end{eqnarray}
where the last equality is obtained via the change of variables $u=s\cdot n^{d/(d+1)}$, and the conclusion by the bounded convergence theorem.
Altogether, we get that
\begin{eqnarray}
\mathbb{E}\left[\#\left(\mathbb{T}_n^d\middle\backslash\mathcal{B}_n^{p+1}(k_n)\right)\right]&=&\sum_{x\in\mathbb{T}_n^d}\mathbb{P}\left(x\notin\mathcal{B}_n^{p+1}(k_n)\right)\notag\\
&\underset{n\to\infty}{\sim}&\sum_{x\in\mathbb{T}_n^d}\exp\left(-\sum_{i=1}^{k_n}\frac{\#B_i^{x,x}}{n^d}\cdot f_p\left(\frac{i-1}{n^{d/(d+1)}}\right)\right)\notag\\
&=&n^d\cdot\exp\left(-\sum_{i=1}^{k_n}\frac{\#\overline{B}(0,k_n-i)}{n^d}\cdot f_p\left(\frac{i-1}{n^{d/(d+1)}}\right)\right)\label{eq:keylem1stmomentbis}\\
&\sim&n^d\cdot\exp\left(-\frac{2^d}{d!}\cdot\int_0^t(t-s)^d\cdot f_p(s)\mathrm{d}s\right)=n^d\cdot\frac{1}{f_{p+1}(t)}.\label{eq:keylem1stmoment}
\end{eqnarray}
Let us now move on to the second moment.
By \eqref{eq:keylemestimate}, we obtain that
\begin{eqnarray*}
\mathbb{E}\left[\#\left(\mathbb{T}_n^d\middle\backslash\mathcal{B}_n^{p+1}(k_n)\right)^2\right]&=&\sum_{x_1,x_2\in\mathbb{T}_n^d}\mathbb{P}\left(x_1,x_2\notin\mathcal{B}_n^{p+1}(k_n)\right)\\
&\underset{n\to\infty}{\sim}&\sum_{x_1,x_2\in\mathbb{T}_n^d}\exp\left(-\sum_{i=1}^{k_n}\frac{\#B_i^{x_1,x_2}}{n^d}\cdot f_p\left(\frac{i-1}{n^{d/(d+1)}}\right)\right)\\
&=&n^d\cdot\sum_{x\in\mathbb{T}_n^d}\exp\left(-\sum_{i=1}^{k_n}\frac{\#B_i^{0,x}}{n^d}\cdot f_p\left(\frac{i-1}{n^{d/(d+1)}}\right)\right).
\end{eqnarray*}
Thus, using \eqref{eq:keylem1stmomentbis}, we get
\begin{eqnarray*}
\lefteqn{\frac{\mathbb{E}\left[\#\left(\mathbb{T}_n^d\middle\backslash\mathcal{B}_n^{p+1}(k_n)\right)^2\right]}{\mathbb{E}\left[\#\left(\mathbb{T}_n^d\middle\backslash\mathcal{B}_n^{p+1}(k_n)\right)\right]^2}}\\
&\underset{n\to\infty}{\sim}&\frac{1}{n^d}\cdot\sum_{x\in\mathbb{T}_n^d}\exp\left(\sum_{i=1}^{k_n}\frac{\#\left(\overline{B}(0,k_n-i)\cap\overline{B}(x,k_n-i)\right)}{n^d}\cdot f_p\left(\frac{i-1}{n^{d/(d+1)}}\right)\right).
\end{eqnarray*}
Then, we upper bound the last term as follows: since $\overline{B}(0,k_n-i)\cap\overline{B}(x,k_n-i)=\emptyset$ as soon as $d(0,x)>2k_n$, we have
\begin{eqnarray*}
\lefteqn{\frac{1}{n^d}\cdot\sum_{x\in\mathbb{T}_n^d}\exp\left(\sum_{i=1}^{k_n}\frac{\#\left(\overline{B}(0,k_n-i)\cap\overline{B}(x,k_n-i)\right)}{n^d}\cdot f_p\left(\frac{i-1}{n^{d/(d+1)}}\right)\right)}\\
&\leq&1+\frac{1}{n^d}\cdot\#\overline{B}(0,2k_n)\cdot\exp\left(\sum_{i=1}^{k_n}\frac{\#\overline{B}(0,k_n-i)}{n^d}\cdot f_p\left(\frac{i-1}{n^{d/(d+1)}}\right)\right).
\end{eqnarray*}
Finally, by \eqref{eq:untag}, and since $\#\overline{B}(0,2k_n)\cdot n^{-d}\asymp n^{-d/(d+1)}\rightarrow0$ as $n\to\infty$, we obtain that
\begin{equation*}
\frac{\mathbb{E}\left[\#\left(\mathbb{T}_n^d\middle\backslash\mathcal{B}_n^{p+1}(k_n)\right)^2\right]}{\mathbb{E}\left[\#\left(\mathbb{T}_n^d\middle\backslash\mathcal{B}_n^{p+1}(k_n)\right)\right]^2}\longrightarrow1\quad\text{as $n\to\infty$,}
\end{equation*}
which together with \eqref{eq:keylem1stmoment} yields \eqref{eq:convoitedconv}, by the Bienaymé--Chebyshev inequality.

To complete the proof of the proposition, it remains to prove \eqref{eq:keylemestimate}.
For each $x_1,x_2\in\mathbb{T}_n^d$, we set
\[A_i^{x_1,x_2}=\left(Y_{i}^{p+1}\notin B_i^{x_1,x_2}\right)\quad\text{for all $i\in\llbracket1,k_n\rrbracket$,}\]
so that
\[\left(x_1,x_2\notin\mathcal{B}_n^{p+1}(k_n)\right)=\bigcap_{i=1}^{k_n}A_i^{x_1,x_2}.\]
By Claim \ref{claim:rejectionsamplingp}, the following holds: for each $i\in\llbracket1,k_n\rrbracket$, we have $A_i^{x_1,x_2}\in\mathcal{F}_i$.
Moreover, on the event $\left(\mathcal{B}_n^p(i-1)\varsubsetneq\mathbb{T}_n^d\right)$, since the conditional distribution of $Y_i^{p+1}$ given $\mathcal{F}_{i-1}$ is the uniform distribution on the complement of $\mathcal{B}_n^p(i-1)$, we have
\[\mathbb{P}\left(A_i^{x_1,x_2}\,\middle|\,\mathcal{F}_{i-1}\right)=1-\frac{\#\left(B_i^{x_1,x_2}\middle\backslash\mathcal{B}_n^p(i-1)\right)}{\#\left(\mathbb{T}_n^d\middle\backslash\mathcal{B}_n^p(i-1)\right)}\quad\text{almost surely.}\]
Now, observe that for each $j\in\llbracket1,k_n\rrbracket$, on the event $\bigcap_{i=1}^jA_i^{x_1,x_2}$, we have
\begin{equation}\label{eq:nested}
B_j^{x_1,x_2}\subset\left.\mathbb{T}_n^d\middle\backslash\mathcal{B}_n^{p+1}(j)\right.\subset\left.\mathbb{T}_n^d\middle\backslash\mathcal{B}_n^p(j-1)\right..
\end{equation}
The second inclusion holds by virtue of Claim \ref{claim:inclusionintermediate}, and we now justify the first inclusion.
For all $i\in\llbracket1,j\rrbracket$, by the triangle inequality, we have
\[Y_i^{p+1}\notin B_i^{x_1,x_2}\Longrightarrow\overline{B}\left(Y_i^{p+1},j-i\right)\subset\left.\mathbb{T}_n^d\middle\backslash B_j^{x_1,x_2}\right..\]
Therefore, we get that
\[\bigcap_{i=1}^j\left(Y_i^{p+1}\notin B_i^{x_1,x_2}\right)\subset\left(\bigcup_{i=1}^j\overline{B}\left(Y_i^{p+1},j-i\right)\subset\left.\mathbb{T}_n^d\middle\backslash B_j^{x_1,x_2}\right.\right),\]
which yields the desired result: on the event $\bigcap_{i=1}^jA_i^{x_1,x_2}$, we have $\mathcal{B}_n^{p+1}(j)\subset\left.\mathbb{T}_n^d\middle\backslash B_j^{x_1,x_2}\right.$, hence the first inclusion of \eqref{eq:nested} holds.
It follows from \eqref{eq:nested} that on the event $\bigcap_{i=1}^{k_n}A_i^{x_1,x_2}$, for all $i\in\llbracket1,k_n\rrbracket$, we have
\begin{equation}\label{eq:nested'}
\mathbb{P}\left(A_i^{x_1,x_2}\,\middle|\,\mathcal{F}_{i-1}\right)=1-\frac{\#B_i^{x_1,x_2}}{\#\left(\mathbb{T}_n^d\middle\backslash\mathcal{B}_n^p(i-1)\right)}=1-\frac{\#B_i^{x_1,x_2}\cdot n^{-d}}{\#\left(\mathbb{T}_n^d\middle\backslash\mathcal{B}_n^p(i-1)\right)\cdot n^{-d}}\quad\text{almost surely.}
\end{equation}
From there, we are tempted to write
\begin{equation}\label{eq:mart}
\mathbb{E}\left[\prod_{i=1}^{k_n}\left(1-\frac{\#B_i^{x_1,x_2}\cdot n^{-d}}{\#\left(\mathbb{T}_n^d\middle\backslash\mathcal{B}_n^p(i-1)\right)\cdot n^{-d}}\right)^{-1}\cdot\mathbf{1}\left(\bigcap_{i=1}^{k_n}A_i^{x_1,x_2}\right)\right]=\mathbb{E}\left[\prod_{i=1}^{k_n}\frac{\mathbf{1}\left(A_i^{x_1,x_2}\right)}{\mathbb{P}\left(A_i^{x_1,x_2}\,\middle|\,\mathcal{F}_{i-1}\right)}\right]=1.
\end{equation}
Indeed, assume for a second that this martingale identity holds.\footnote{Observe that under the assumption that $\mathbb{P}\left(A_i^{x_1,x_2}\,\middle|\,\mathcal{F}_{i-1}\right)>0$ almost surely for all $i\in\llbracket1,k_n\rrbracket$, the process
\[\left(M_j:=\prod_{i=1}^j\frac{\mathbf{1}\left(A_i^{x_1,x_2}\right)}{\mathbb{P}\left(A_i^{x_1,x_2}\,\middle|\,\mathcal{F}_{i-1}\right)}\right)_{j\in\llbracket0,k_n\rrbracket}\]
is a martingale with respect to the filtration $(\mathcal{F}_j)_{j\in\llbracket0,k_n\rrbracket}$.}
Informally, one could then try to derive the desired estimate \eqref{eq:keylemestimate} as follows:
\begin{eqnarray}
1&=&\mathbb{E}\left[\prod_{i=1}^{k_n}\left(1-\frac{\#B_i^{x_1,x_2}\cdot n^{-d}}{\#\left(\mathbb{T}_n^d\middle\backslash\mathcal{B}_n^p(i-1)\right)\cdot n^{-d}}\right)^{-1}\cdot\mathbf{1}\left(\bigcap_{i=1}^{k_n}A_i^{x_1,x_2}\right)\right]\notag\\
&\approx&\mathbb{E}\left[\exp\left(\sum_{i=1}^{k_n}\frac{\#B_i^{x_1,x_2}\cdot n^{-d}}{\#\left(\mathbb{T}_n^d\middle\backslash\mathcal{B}_n^p(i-1)\right)\cdot n^{-d}}\right)\cdot\mathbf{1}\left(\bigcap_{i=1}^{k_n}A_i^{x_1,x_2}\right)\right]\notag\\
&\approx&\exp\left(\sum_{i=1}^{k_n}\frac{\#B_i^{x_1,x_2}}{n^d}\cdot f_p\left(\frac{i-1}{n^{d/(d+1)}}\right)\right)\cdot\mathbb{P}\left(\bigcap_{i=1}^{k_n}A_i^{x_1,x_2}\right).\label{eq:onaurait}
\end{eqnarray}
However, in order to formalise \eqref{eq:onaurait}, and to get the martingale identity \eqref{eq:mart} in the first place, one would need some control on the terms
\[\mathbb{P}\left(A_i^{x_1,x_2}\,\middle|\,\mathcal{F}_{i-1}\right)\quad\text{and}\quad1-\frac{\#B_i^{x_1,x_2}\cdot n^{-d}}{\#\left(\mathbb{T}_n^d\middle\backslash\mathcal{B}_n^p(i-1)\right)\cdot n^{-d}},\]
which is not directly available a priori.
For this reason, we need to implement the following truncation argument, which is reminiscent of the proof of the martingale central limit theorem (see, e.g, \cite[Chapter 15]{martingaleclt}).
For each $i\in\llbracket1,k_n\rrbracket$, consider the good event
\[G_i=\left(\frac{\#\left(\mathbb{T}_n^d\middle\backslash\mathcal{B}^p_n(i-1)\right)}{n^d}\geq\frac{1/2}{\sup_{s\in[0,t]}f_p(s)}\right).\]
Note that $G_i\in\mathcal{F}_{i-1}$, and that on $G_i$, almost surely, we have
\[\mathbb{P}\left(A_i^{x_1,x_2}\,\middle|\,\mathcal{F}_{i-1}\right)=1-\frac{\#\left(B_i^{x_1,x_2}\middle\backslash\mathcal{B}_n^p(i-1)\right)\cdot n^{-d}}{\#\left(\mathbb{T}_n^d\middle\backslash\mathcal{B}^p_n(i-1)\right)\cdot n^{-d}}\geq1-\frac{2\cdot\#\overline{B}(0,k_n)\cdot n^{-d}}{\left.(1/2)\middle/\left(\sup_{s\in[0,t]}f_p(s)\right)\right.}.\]
Then, consider the following truncated random variable:
\[W_i^{x_1,x_2}=\begin{cases}
\mathbf{1}\left(A_i^{x_1,x_2}\right)&\text{on $G_i$,}\\
1&\text{otherwise.}
\end{cases}\]
Almost surely, we have
\[\mathbb{E}\left[W_i^{x_1,x_2}\,\middle|\,\mathcal{F}_{i-1}\right]=\mathbb{P}\left(A_i^{x_1,x_2}\,\middle|\,\mathcal{F}_{i-1}\right)\cdot\mathbf{1}(G_i)+1\cdot(1-\mathbf{1}(G_i))\geq1-\frac{2\cdot\#\overline{B}(0,k_n)\cdot n^{-d}}{\left.(1/2)\middle/\left(\sup_{s\in[0,t]}f_p(s)\right)\right.}.\]
In particular, for all $n\in\mathbb{N}^*$ sufficiently large so that
\[\frac{2\cdot\#\overline{B}(0,k_n)\cdot n^{-d}}{\left.(1/2)\middle/\left(\sup_{s\in[0,t]}f_p(s)\right)\right.}\leq\frac{1}{2},\]
we have $\mathbb{E}\left[W_i^{x_1,x_2}\,\middle|\,\mathcal{F}_{i-1}\right]\geq1/2>0$ for all $i\in\llbracket1,k_n\rrbracket$.
Therefore, the process
\[\left(N_j:=\prod_{i=1}^j\frac{W_i^{x_1,x_2}}{\mathbb{E}\left[W_i^{x_1,x_2}\,\middle|\,\mathcal{F}_{i-1}\right]}\right)_{j\in\llbracket0,k_n\rrbracket}\]
is a martingale with respect to the filtration $\left(\mathcal{F}_j\right)_{j\in\llbracket0,k_n\rrbracket}$.
In particular, we get that
\begin{equation}\label{eq:marttrunc}
\mathbb{E}\left[\prod_{i=1}^{k_n}\frac{W_i^{x_1,x_2}}{\mathbb{E}\left[W_i^{x_1,x_2}\,\middle|\,\mathcal{F}_{i-1}\right]}\right]=\mathbb{E}\left[N_{k_n}\right]=\mathbb{E}[N_0]=\mathbb{E}[1]=1.
\end{equation}
Armed with this martingale identity, we are now in position to formalise \eqref{eq:onaurait}, in order to derive the desired estimate \eqref{eq:keylemestimate}.
Using \eqref{eq:marttrunc} and the triangle inequality, we write
\begin{eqnarray}
\lefteqn{\left|\exp\left(\sum_{i=1}^{k_n}\frac{\#B_i^{x_1,x_2}}{n^d}\cdot f_p\left(\frac{i-1}{n^{d/(d+1)}}\right)\right)\cdot\mathbb{P}\left(\bigcap_{i=1}^{k_n}A_i^{x_1,x_2}\right)-1\right|}\notag\\
&=&\left|\exp\left(\sum_{i=1}^{k_n}\frac{\#B_i^{x_1,x_2}}{n^d}\cdot f_p\left(\frac{i-1}{n^{d/(d+1)}}\right)\right)\cdot\left(\mathbb{P}\left(\bigcap_{i=1}^{k_n}A_i^{x_1,x_2}\right)- \mathbb{E} \left[ \prod_{i=1}^{k_n}W_i^{x_1,x_2} \right ]\right)\right.\notag\\
&&\left.+\mathbb{E}\left[\left(\exp\left(\sum_{i=1}^{k_n}\frac{\#B_i^{x_1,x_2}}{n^d}\cdot f_p\left(\frac{i-1}{n^{d/(d+1)}}\right)\right)-\prod_{i=1}^{k_n}\mathbb{E}\left[W_i^{x_1,x_2}\,\middle|\,\mathcal{F}_{i-1}\right]^{-1}\right)\cdot\prod_{i=1}^{k_n}W_i^{x_1,x_2}\right]\right|\notag \\
&\leq&\exp\left(\sum_{i=1}^{k_n}\frac{\#B_i^{x_1,x_2}}{n^d}\cdot f_p\left(\frac{i-1}{n^{d/(d+1)}}\right)\right)\cdot\left|\mathbb{P}\left(\bigcap_{i=1}^{k_n}A_i^{x_1,x_2}\right)-\mathbb{E}\left[\prod_{i=1}^{k_n}W_i^{x_1,x_2}\right]\right|\label{eq:lelem1stterm}\\
&&+\mathbb{E}\left[\left|\exp\left(\sum_{i=1}^{k_n}\frac{\#B_i^{x_1,x_2}}{n^d}\cdot f_p\left(\frac{i-1}{n^{d/(d+1)}}\right)\right)-\prod_{i=1}^{k_n}\mathbb{E}\left[W_i^{x_1,x_2}\,\middle|\,\mathcal{F}_{i-1}\right]^{-1}\right|\cdot\prod_{i=1}^{k_n}W_i^{x_1,x_2}\right].\label{eq:lelem2ndterm}
\end{eqnarray}
Now, let us control each term separately.
First, since $\mathbf{1}\left(\bigcap_{i=1}^{k_n}A_i^{x_1,x_2}\right)$ and $\prod_{i=1}^{k_n}W_i^{x_1,x_2}$ take values in $\{0,1\}$ and agree on the good event $G:=\bigcap_{i=1}^{k_n}G_i$, the first term (at \eqref{eq:lelem1stterm}) is at most
\begin{eqnarray*}
\exp\left(k_n\cdot\frac{2\cdot\#\overline{B}(0,k_n)}{n^d}\cdot\sup_{s\in[0,t]}f_p(s)\right)\cdot\left(1-\mathbb{P}(G)\right).
\end{eqnarray*}
This upper bound does not depend on $x_1,x_2$, and goes to zero as $n\to\infty$, since
\[\begin{split}
\mathbb{P}(G)&\geq\mathbb{P}\left(\bigcap_{i=1}^{k_n}\left(\frac{\#\left(\mathbb{T}_n^d\middle\backslash\mathcal{B}_n^p(i-1)\right)}{n^d}\geq\frac{1}{f_p\left((i-1)\cdot n^{-d/(d+1)}\right)}-\frac{1/2}{\sup_{s\in[0,t]}f_p(s)}\right)\right)\\
&\geq\mathbb{P}\left(\sup_{s\in[0,t]}\left|\frac{\#\left(\mathbb{T}_n^d\middle\backslash\mathcal{B}_n^p\left(\left\lfloor s\cdot n^{d/(d+1)}\right\rfloor\right)\right)}{n^d}-\frac{1}{f_p(s)}\right|\leq\frac{1/2}{\sup_{s\in[0,t]}f_p(s)}\right)\underset{n\to\infty}{\longrightarrow}1,
\end{split}\]
where the last convergence holds by the induction hypothesis.
Next, to control the second term (at \eqref{eq:lelem2ndterm}), we split the expectation into two parts.
First, on the complement of $G$, we simply bound the integrand by
\[\exp\left(k_n\cdot\frac{2\cdot\#\overline{B}(0,k_n)}{n^d}\cdot\sup_{s\in[0,t]}f_p(s)\right)+\prod_{i=1}^{k_n}\left(1-\frac{2\cdot\#\overline{B}(0,k_n)\cdot n^{-d}}{\left.(1/2)\middle/\left(\sup_{s\in[0,t]}f_p(s)\right)\right.}\right)^{-1},\]
which does not depend on $x_1,x_2$, and remains bounded as $n\to\infty$.
Since $\mathbb{P}(G)\rightarrow1$ as $n\to\infty$, this part of the expectation vanishes uniformly in $x_1,x_2\in\mathbb{T}_n^d$, as desired.
To control the other part, on the good event $G$, we estimate the integrand as follows:
\begin{eqnarray*}
\lefteqn{\left|\exp\left(\sum_{i=1}^{k_n}\frac{\#B_i^{x_1,x_2}}{n^d}\cdot f_p\left(\frac{i-1}{n^{d/(d+1)}}\right)\right)-\prod_{i=1}^{k_n}\mathbb{E}\left[W_i^{x_1,x_2}\,\middle|\,\mathcal{F}_{i-1}\right]^{-1}\right|\cdot\prod_{i=1}^{k_n}W_i^{x_1,x_2}}\\
&=&\left|\exp\left(\sum_{i=1}^{k_n}\frac{\#B_i^{x_1,x_2}}{n^d}\cdot f_p\left(\frac{i-1}{n^{d/(d+1)}}\right)\right)-\prod_{i=1}^{k_n}\mathbb{P}\left(A_i^{x_1,x_2}\,\middle|\,\mathcal{F}_{i-1}\right)^{-1}\right|\cdot\mathbf{1}\left(\bigcap_{i=1}^{k_n}A_i^{x_1,x_2}\right)\\
&=&\left|\exp\left(\sum_{i=1}^{k_n}\frac{\#B_i^{x_1,x_2}}{n^d}\cdot f_p\left(\frac{i-1}{n^{d/(d+1)}}\right)\right)-\prod_{i=1}^{k_n}\left(1-\frac{\#B_i^{x_1,x_2}\cdot n^{-d}}{\#\left(\mathbb{T}_n^d\middle\backslash\mathcal{B}_n^p(i-1)\right)\cdot n^{-d}}\right)^{-1}\right|\\
&&\cdot\mathbf{1}\left(\bigcap_{i=1}^{k_n}A_i^{x_1,x_2}\right)\\
&\leq&\left|\exp\left(\sum_{i=1}^{k_n}\frac{\#B_i^{x_1,x_2}}{n^d}\cdot f_p\left(\frac{i-1}{n^{d/(d+1)}}\right)\right)-\prod_{i=1}^{k_n}\left(1-\frac{\#B_i^{x_1,x_2}\cdot n^{-d}}{\#\left(\mathbb{T}_n^d\middle\backslash\mathcal{B}_n^p(i-1)\right)\cdot n^{-d}}\right)^{-1}\right|.
\end{eqnarray*}
The first equality holds by the definition of the truncated random variables $W_i^{x_1,x_2}$, and the second equality holds by \eqref{eq:nested'}.
Now, recall that on $G$, we have
\[\frac{\#B_i^{x_1,x_2}\cdot n^{-d}}{\#\left(\mathbb{T}_n^d\middle\backslash\mathcal{B}_n^p(i-1)\right)\cdot n^{-d}}\leq\frac{2\cdot\#\overline{B}(0,k_n)\cdot n^{-d}}{\left.(1/2)\middle/\left(\sup_{s\in[0,t]}f_p(s)\right)\right.}\leq\frac{1}{2}\quad\text{for all $i\in\llbracket1,k_n\rrbracket$,}\]
where the last inequality holds by assumption on $n$.
Thus, we may write
\[\prod_{i=1}^{k_n}\left(1-\frac{\#B_i^{x_1,x_2}\cdot n^{-d}}{\#\left(\mathbb{T}_n^d\middle\backslash\mathcal{B}_n^p(i-1)\right)\cdot n^{-d}}\right)^{-1}=\exp\left(\sum_{i=1}^{k_n}\frac{\#B_i^{x_1,x_2}\cdot n^{-d}}{\#\left(\mathbb{T}_n^d\middle\backslash\mathcal{B}_n^p(i-1)\right)\cdot n^{-d}}+\Delta_n^{x_1,x_2}\right),\]
where
\[\Delta_n^{x_1,x_2}=\sum_{i=1}^{k_n}\left(-\ln\left(1-\frac{\#B_i^{x_1,x_2}\cdot n^{-d}}{\#\left(\mathbb{T}_n^d\middle\backslash\mathcal{B}_n^p(i-1)\right)\cdot n^{-d}}\right)-\frac{\#B_i^{x_1,x_2}\cdot n^{-d}}{\#\left(\mathbb{T}_n^d\middle\backslash\mathcal{B}_n^p(i-1)\right)\cdot n^{-d}}\right).\]
Then, we may bound $\Delta_n^{x_1,x_2}$ by using the inequality $|-\ln(1-u)-u|\leq u^2$ for all $u\in[0,1/2]$: we have
\[\left|\Delta_n^{x_1,x_2}\right|\leq\sum_{i=1}^{k_n}\left(\frac{\#B_i^{x_1,x_2}\cdot n^{-d}}{\#\left(\mathbb{T}_n^d\middle\backslash\mathcal{B}_n^p(i-1)\right)\cdot n^{-d}}\right)^2\leq k_n\cdot\left(\frac{2\cdot\#\overline{B}(0,k_n)\cdot n^{-d}}{\left.(1/2)\middle/\left(\sup_{s\in[0,t]}f_p(s)\right)\right.}\right)^2.\]
Now, to complete the proof of \eqref{eq:keylemestimate}, it remains to show that as $n\to\infty$, we have
\[\max_{x_1,x_2\in\mathbb{T}_n^d}\mathbb{E}\left[\left|\exp\left(U_n^{x_1,x_2}\right)-\exp\left(V_n^{x_1,x_2}\right)\right|\cdot\mathbf{1}(G)\right]\longrightarrow0,\]
where
\[U_n^{x_1,x_2}=\sum_{i=1}^{k_n}\frac{\#B_i^{x_1,x_2}}{n^d}\cdot f_p\left(\frac{i-1}{n^{d/(d+1)}}\right)\quad\text{and}\quad V_n^{x_1,x_2}=\sum_{i=1}^{k_n}\frac{\#B_i^{x_1,x_2}\cdot n^{-d}}{\#\left(\mathbb{T}_n^d\middle\backslash\mathcal{B}_n^p(i-1)\right)\cdot n^{-d}}+\Delta_n^{x_1,x_2}.\]
This is a straightforward consequence of the following two facts:
\begin{itemize}
\item  On the good event $G$, we have
\[U_n^{x_1,x_2}\leq k_n\cdot\frac{2\cdot\#\overline{B}(0,k_n)}{n^d}\cdot\sup_{s\in[0,t]}f_p(s),\]
and
\[V_n^{x_1,x_2}\leq k_n\cdot\frac{2\cdot\#\overline{B}(0,k_n)\cdot n^{-d}}{\left.(1/2)\middle/\left(\sup_{s\in[0,t]}f_p(s)\right)\right.}+k_n\cdot\left(\frac{2\cdot\#\overline{B}(0,k_n)\cdot n^{-d}}{\left.(1/2)\middle/\left(\sup_{s\in[0,t]}f_p(s)\right)\right.}\right)^2,\]
where both right-hand sides do not depend on $x_1,x_2$, and remain bounded as $n\to\infty$.

\item On the good event $G$, we have $\left|U_n^{x_1,x_2}-V_n^{x_1,x_2}\right|\leq\Delta_n$, where $\Delta_n$ does not depend on $x_1,x_2$, and converges to $0$ in probability as $n\to\infty$.
Indeed, we have
\[U_n^{x_1,x_2}-V_n^{x_1,x_2}=\sum_{i=1}^{k_n}\frac{\#B_i^{x_1,x_2}}{n^d}\cdot\left(f_p\left(\frac{i-1}{n^{d/(d+1)}}\right)-\frac{1}{\#\left(\mathbb{T}_n^d\middle\backslash\mathcal{B}_n^p(i-1)\right)\cdot n^{-d}}\right)-\Delta_n^{x_1,x_2},\]
hence
\begin{eqnarray*}
\left|U_n^{x_1,x_2}-V_n^{x_1,x_2}\right|&\leq&k_n\cdot\frac{2\cdot\#\overline{B}(0,k_n)}{n^d}\cdot\sup_{s\in[0,t]}\left|f_p(s)-\frac{1}{\#\left(\mathbb{T}_n^d\middle\backslash\mathcal{B}_n^p\left(\left\lfloor s\cdot n^{d/(d+1)}\right\rfloor\right)\right)\cdot n^{-d}}\right|\\
&&+k_n\cdot\left(\frac{2\cdot\#\overline{B}(0,k_n)\cdot n^{-d}}{\left.(1/2)\middle/\left(\sup_{s\in[0,t]}f_p(s)\right)\right.}\right)^2\\
&=:&\Delta_n\underset{n\to\infty}{\overset{\mathbb{P}}{\longrightarrow}}0.
\end{eqnarray*}
\end{itemize}
Note that the last convergence holds by the induction hypothesis.
We have thus established \eqref{eq:keylemestimate}, which concludes the proof of the proposition.
\end{proof}

Now armed with Proposition \ref{prop:limitingdensityBnp}, we may prove the desired functional convergence \eqref{eq:functionalconvergence}.

\begin{proof}[Proof of \eqref{eq:functionalconvergence}]
As mentioned previously, we will use the coupling of the rejection sampling burning process with the intermediate burning processes $\left(\mathcal{B}_n^p(\cdot)\,;\,p\in\mathbb{N}^*\right)$: let us show that
\[\sup_{t\in[0,T[}\left|\frac{\#\left(\mathbb{T}_n^d\middle\backslash\mathcal{B}_n^*\left(\left\lfloor t\cdot n^{d/(d+1)}\right\rfloor\right)\right)}{n^d}-\frac{1}{y^{(d)}(t)}\right|\overset{\mathbb{P}}{\longrightarrow}0\quad\text{as $n\to\infty$.}\]
By a probabilistic version of Dini's theorem (see, e.g, \cite[Lemma 11.6]{chef}), it suffices to prove that for each $t\in[0,T[$, we have\footnote{Indeed, since $y^{(d)}(t)\rightarrow\infty$ as $t\to T^-$, we can extend $f:t\in{[0,T[}\mapsto1/y^{(d)}(t)$ continuously to $\mathbb{R}_+$ by setting $f(t)=0$ for all $t\in[T,\infty[$.
Then, \eqref{eq:goalfunctionalconvergence} easily extends to $\mathbb{R}_+$, which readily yields \eqref{eq:functionalconvergence}, by \cite[Lemma 11.6]{chef}.}
\begin{equation}\label{eq:goalfunctionalconvergence}
\frac{\#\left(\mathbb{T}_n^d\middle\backslash\mathcal{B}_n^*\left(\left\lfloor t\cdot n^{d/(d+1)}\right\rfloor\right)\right)}{n^d}\overset{\mathbb{P}}{\longrightarrow}\frac{1}{y^{(d)}(t)}\quad\text{as $n\to\infty$.}
\end{equation}
For each $t\in[0,T[$, set $k_n(t)=\left\lfloor t\cdot n^{d/(d+1)}\right\rfloor$ for all $n\in\mathbb{N}^*$, for short.
For each $p\in\mathbb{N}$, we let
\[\zeta_p(t)=\varlimsup_{n\to\infty}\mathbb{E}\left[\frac{\#\left(\mathcal{B}_n^*(k_n(t))\middle\backslash\mathcal{B}_n^p(k_n(t))\right)}{n^d}\right]\quad\text{for all $t\in[0,T[$.}\]
We claim that it suffices to prove that for each $t\in[0,T[$, we have $\zeta_p(t)\rightarrow0$ as $p\to\infty$.
Indeed, by the triangle inequality, we have
\begin{eqnarray*}
\lefteqn{\left|\frac{\#\left(\mathbb{T}_n^d\middle\backslash\mathcal{B}_n^*(k_n(t))\right)}{n^d}-\frac{1}{y^{(d)}(t)}\right|}\\
&\leq&\frac{\#\left(\mathcal{B}_n^*(k_n(t))\middle\backslash\mathcal{B}_n^p(k_n(t))\right)}{n^d}+\left|\frac{\#\left(\mathbb{T}_n^d\middle\backslash\mathcal{B}_n^p(k_n(t))\right)}{n^d}-\frac{1}{f_p(t)}\right|+\left(\frac{1}{f_p(t)}-\frac{1}{y^{(d)}(t)}\right),
\end{eqnarray*}
hence
\[\varlimsup_{n\to\infty}\mathbb{E}\left[\left|\frac{\#\left(\mathbb{T}_n^d\middle\backslash\mathcal{B}_n^*(k_n(t))\right)}{n^d}-\frac{1}{y^{(d)}(t)}\right|\right]\leq\zeta_p(t)+\left(\frac{1}{f_p(t)}-\frac{1}{y^{(d)}(t)}\right),\]
where we use Proposition \ref{prop:limitingdensityBnp} to handle the second term.
By Claim \ref{claim:picard}, the last term goes to $0$ as $p\to\infty$, and therefore it suffices to show that $\zeta_p(t)\rightarrow0$, as claimed above.
In this direction, we will establish the following inequality: for each $p\in\mathbb{N}$, we have
\begin{equation}\label{eq:durham}
\zeta_{p+1}(t)\leq\frac{2^d}{d!}\cdot\int_0^t(t-s)^d\cdot f_p(s)\cdot\zeta_p(s)\mathrm{d}s\quad\text{for all $t\in[0,T[$.}
\end{equation}
To prove \eqref{eq:durham}, recall that for all $k\in\mathbb{N}$, we have (see Subsection \ref{subsec:coupling})
\[\mathcal{B}_n^*(k)=\bigcup_{i=1}^k\overline{B}\left(Y_i^*,k-i\right)\quad\text{and}\quad\mathcal{B}_n^{p+1}(k)=\bigcup_{i=1}^k\overline{B}\left(Y_i^{p+1},k-i\right).\]
In particular, we have
\[\left.\mathcal{B}_n^*(k_n(t))\middle\backslash\mathcal{B}_n^{p+1}(k_n(t))\right.\subset\bigcup_{\substack{i=1\\Y_i^*\neq Y_i^{p+1}}}^{k_n(t)}\overline{B}\left(Y_i^*,k_n(t)-i\right).\]
This yields
\[\begin{split}
\frac{\#\left(\mathcal{B}_n^*(k_n(t))\middle\backslash\mathcal{B}_n^{p+1}(k_n(t))\right)}{n^d}&\leq\sum_{i=1}^{k_n(t)}\frac{\#\overline{B}\left(Y_i^*,k_n(t)-i\right)}{n^d}\cdot\mathbf{1}\left(Y_i^*\neq Y_i^{p+1}\right)\\
&\leq\sum_{i=1}^{k_n(t)}\frac{\#\overline{B}(0,k_n(t)-i)}{n^d}\cdot\mathbf{1}\left(Y_i^{p+1}\in\left.\mathcal{B}_n^*(i-1)\middle\backslash\mathcal{B}_n^p(i-1)\right.\right).
\end{split}\]
Indeed, with the notation introduced in Subsection \ref{subsec:coupling}, the following holds: for each $i\in\llbracket1,k_n(t)\rrbracket$, on the event $\left(Y_i^*\neq Y_i^{p+1}\right)$, we must have $H_i^*>H_i^{p+1}$ (recall Claim \ref{claim:inclusionintermediate}), hence $H_i^{p+1}<\infty$ and
\[Y_i^{p+1}=X_i^{H_i^{p+1}}\in\left.\mathcal{B}_n^*(i-1)\middle\backslash\mathcal{B}_n^p(i-1)\right..\]
Now, recall from Claim \ref{claim:rejectionsamplingp} that the law of $Y_i^{p+1}$ conditionally on $\mathcal{F}_{i-1}$ is the uniform distribution on the complement of $\mathcal{B}_n^p(i-1)$, provided that $\mathcal{B}_n^p(i-1)\varsubsetneq\mathbb{T}_n^d$.
We deduce that
\[\mathbb{P}\left(Y_i^{p+1}\in\left.\mathcal{B}_n^*(i-1)\middle\backslash\mathcal{B}_n^p(i-1)\right.\,\middle|\,\mathcal{F}_{i-1}\right)\leq Z_n^p(i-1)\quad\text{almost surely,}\]
where
\[Z_n^p(i-1)=\begin{cases}
\frac{\#\left(\mathcal{B}_n^*(i-1)\middle\backslash\mathcal{B}_n^p(i-1)\right)\cdot n^{-d}}{\#\left(\mathbb{T}_n^d\middle\backslash\mathcal{B}_n^p(i-1)\right)\cdot n^{-d}}&\text{if $\mathcal{B}_n^p(i-1)\varsubsetneq\mathbb{T}_n^d$,}\\
1&\text{otherwise.}
\end{cases}\]
Summing up, we get that
\begin{eqnarray*}
\lefteqn{\mathbb{E}\left[\frac{\#\left(\mathcal{B}_n^*(k_n(t))\middle\backslash\mathcal{B}_n^{p+1}(k_n(t))\right)}{n^d}\right]}\\
&\leq&\sum_{i=1}^{k_n(t)}\frac{\#\overline{B}(0,k_n(t)-i)}{n^d}\cdot\mathbb{E}\left[Z_n^p(i-1)\right]\\
&=&\int_0^{k_n(t)}\frac{\#\overline{B}(0,k_n(t)-\lceil u\rceil)}{n^d}\cdot\mathbb{E}\left[Z_n^p(\lfloor u\rfloor)\right]\mathrm{d}u\\
&=&\int_0^{k_n(t)\cdot n^{-d/(d+1)}}\frac{\#\overline{B}\left(0,k_n(t)-\left\lceil s\cdot n^{d/(d+1)}\right\rceil\right)}{\left(n^{d/(d+1)}\right)^d}\cdot\mathbb{E}\left[Z_n^p(k_n(s))\right]\mathrm{d}s,
\end{eqnarray*}
where the last equality follows from the change of variables $u=s\cdot n^{d/(d+1)}$.
We deduce that
\begin{eqnarray*}
\zeta_{p+1}(t)&\leq&\varlimsup_{n\to\infty}\int_0^{k_n(t)\cdot n^{-d/(d+1)}}\frac{\#\overline{B}\left(0,k_n(t)-\left\lceil s\cdot n^{d/(d+1)}\right\rceil\right)}{\left(n^{d/(d+1)}\right)^d}\cdot\mathbb{E}\left[Z_n^p(k_n(s))\right]\mathrm{d}s\\
&\leq&\frac{2^d}{d!}\cdot\int_0^t(t-s)^d\cdot\varlimsup_{n\to\infty}\mathbb{E}\left[Z_n^p(k_n(s))\right]\mathrm{d}s,
\end{eqnarray*}
where we use Fatou's lemma for the second inequality.
This yields \eqref{eq:durham}, provided we prove that for all $t\in[0,T[$, we have
\begin{equation}\label{eq:p13}
\varlimsup_{n\to\infty}\mathbb{E}\left[Z_n^p(k_n(t))\right]=f_p(t)\cdot\zeta_p(t).
\end{equation}
To this end, fix $\varepsilon>0$, and consider the good event
\[G_n=\left(\text{$\mathcal{B}_n^p(k_n(t))\varsubsetneq\mathbb{T}_n^d$ and $\left|\frac{1}{\#\left(\mathbb{T}_n^d\middle\backslash\mathcal{B}_n^p(k_n(t))\right)\cdot n^{-d}}-f_p(t)\right|\leq\varepsilon$}\right).\]
We have
\begin{eqnarray*}
\lefteqn{\left|\mathbb{E}\left[Z_n^p(k_n(t))\right]-f_p(t)\cdot\mathbb{E}\left[\frac{\#\left(\mathcal{B}_n^*(k_n(t))\middle\backslash\mathcal{B}_n^p(k_n(t))\right)}{n^d}\right]\right|}\\
&\leq&\mathbb{E}\left[\left|Z_n^p(k_n(t))-f_p(t)\cdot\frac{\#\left(\mathcal{B}_n^*(k_n(t))\middle\backslash\mathcal{B}_n^p(k_n(t))\right)}{n^d}\right|\right]\\
&\leq&\mathbb{E}\left[\left|\frac{1}{\#\left(\mathbb{T}_n^d\middle\backslash\mathcal{B}_n^p(k_n(t))\right)\cdot n^{-d}}-f_p(t)\right|\cdot\frac{\#\left(\mathcal{B}_n^*(k_n(t))\middle\backslash\mathcal{B}_n^p(k_n(t))\right)}{n^d}\cdot\mathbf{1}(G_n)\right]\\
&&+(1+f_p(t))\cdot(1-\mathbb{P}(G_n))\\
&\leq&\varepsilon+(1+f_p(t))\cdot(1-\mathbb{P}(G_n)),
\end{eqnarray*}
hence
\[\varlimsup_{n\to\infty}\left|\mathbb{E}\left[Z_n^p(k_n(t))\right]-f_p(t)\cdot\mathbb{E}\left[\frac{\#\left(\mathcal{B}_n^*(k_n(t))\middle\backslash\mathcal{B}_n^p(k_n(t))\right)}{n^d}\right]\right|\leq\varepsilon.\]
We thus obtain \eqref{eq:p13}, which concludes the proof of \eqref{eq:durham}.

Finally, it remains to explain how \eqref{eq:durham} implies that for each $t\in[0,T[$, we have $\zeta_p(t)\rightarrow0$ as $p\to\infty$.
Set $\zeta(t)=\varlimsup_{p\to\infty}\zeta_p(t)$ for all $t\in[0,T[$: we want to show that $\zeta(t)=0$ for all $t\in[0,T[$.
To this end, fix $\tau\in[0,T[$.
By Fatou's lemma, we get from \eqref{eq:durham} that
\[\text{$\zeta(t)\leq C\cdot\int_0^t\zeta(s)\mathrm{d}s$ \quad for all $t\in[0,\tau]$, \quad where \quad $C=\frac{2^d}{d!}\cdot\sup_{s\in[0,\tau]}(\tau-s)^d\cdot y^{(d)}(s)\in(0,\infty).$}\]
By the Grönwall lemma, this entails that $\zeta(t)=0$ for all $t\in[0,\tau]$ as desired, provided we prove that ${\zeta:[0,\tau]\rightarrow\mathbb{R}_+}$ is continuous, which we proceed to do now.
Let us first fix $p \in \mathbb{N}$: for all $s\leq t\in[0,\tau]$, we have
\begin{eqnarray*}
|\zeta_p(t)-\zeta_p(s)|&\leq&\varlimsup_{n\to\infty}\left|\mathbb{E}\left[\frac{\#\left(\mathcal{B}_n^*(k_n(t))\middle\backslash\mathcal{B}_n^p(k_n(t))\right)}{n^d}\right]-\mathbb{E}\left[\frac{\#\left(\mathcal{B}_n^*(k_n(s))\middle\backslash\mathcal{B}_n^p(k_n(s))\right)}{n^d}\right]\right|\\
&\leq&\varlimsup_{n\to\infty}\left(\mathbb{E}\left[\frac{\#\left(\mathcal{B}_n^*(k_n(t))\middle\backslash\mathcal{B}_n^*(k_n(s))\right)}{n^d}\right]+\mathbb{E}\left[\frac{\#\left(\mathcal{B}_n^p(k_n(t))\middle\backslash\mathcal{B}_n^p(k_n(s))\right)}{n^d}\right]\right)\\
&=&\varlimsup_{n\to\infty}\mathbb{E}\left[\frac{\#\left(\mathcal{B}_n^*(k_n(t))\middle\backslash\mathcal{B}_n^*(k_n(s))\right)}{n^d}\right]+\left(\frac{1}{f_p(s)}-\frac{1}{f_p(t)}\right),
\end{eqnarray*}
where we use Proposition \ref{prop:limitingdensityBnp} for the last equality.
Now, to handle the first term, we write
\begin{eqnarray*}
\lefteqn{\left.\mathcal{B}_n^*(k_n(t))\middle\backslash\mathcal{B}_n^*(k_n(s))\right.}\\
&\subset&\left(\bigcup_{i=1}^{k_n(s)}\left(\overline{B}\left(Y_i^*,k_n(t)-i\right)\middle\backslash\overline{B}\left(Y_i^*,k_n(s)-i\right)\right)\right)\cup\left(\bigcup_{i=k_n(s)+1}^{k_n(t)}\overline{B}\left(Y_i^*,k_n(t)-i\right)\right).
\end{eqnarray*}
This yields
\begin{eqnarray*}
\lefteqn{\frac{\#\left(\mathcal{B}_n^*(k_n(t))\middle\backslash\mathcal{B}_n^*(k_n(s))\right)}{n^d}}\\
&\leq&\sum_{i=1}^{k_n(s)}\frac{\#\left(\overline{B}\left(Y_i^*,k_n(t)-i\right)\middle\backslash\overline{B}\left(Y_i^*,k_n(s)-i\right)\right)}{n^d}+\sum_{i=k_n(s)+1}^{k_n(t)}\frac{\#\overline{B}\left(Y_i^*,k_n(t)-i\right)}{n^d}\\
&=&\sum_{i=1}^{k_n(s)}\frac{\#\left(\overline{B}\left(0,k_n(t)-i\right)\middle\backslash\overline{B}\left(0,k_n(s)-i\right)\right)}{n^d}+\sum_{i=k_n(s)+1}^{k_n(t)}\frac{\#\overline{B}\left(0,k_n(t)-i\right)}{n^d}.
\end{eqnarray*}
Then, by the same manipulations as we have already performed several times: letting $n\to\infty$, we obtain
\[\varlimsup_{n\to\infty}\mathbb{E}\left[\frac{\#\left(\mathcal{B}_n^*(k_n(t))\middle\backslash\mathcal{B}_n^*(k_n(s))\right)}{n^d}\right]\leq\frac{2^d}{d!}\cdot\int_0^s\left((t-r)^d-(s-r)^d\right)\mathrm{d}r+\frac{2^d}{d!}\cdot\int_s^t(t-r)^d\mathrm{d}r.\]
Thus, we get
\[|\zeta_p(t)-\zeta_p(s)|\leq\frac{2^d}{d!}\cdot\int_0^s\left((t-r)^d-(s-r)^d\right)\mathrm{d}r+\frac{2^d}{d!}\cdot\int_s^t(t-r)^d\mathrm{d}r+\left(\frac{1}{f_p(s)}-\frac{1}{f_p(t)}\right),\]
which entails that $ \zeta$ is continuous, since $|\zeta(t)-\zeta(s)|\leq\varlimsup_{p\to\infty}|\zeta_p(t)-\zeta_p(s)|$.
This completes the proof of \eqref{eq:functionalconvergence}.
\end{proof}

\subsection{Asymptotic behaviour of the random burning number}\label{subsec:asymptoticrandomburning}

In this subsection, we complete the proof of Theorem \ref{thm:rejectionsampling} by proving \eqref{eq:asymptoticrandomburning}.

\begin{proof}[Proof of \eqref{eq:asymptoticrandomburning}]
We start with the lower bound, which is an easy consequence of \eqref{eq:functionalconvergence}: for each $\varepsilon\in(0,T)$, we have
\begin{eqnarray*}
\lefteqn{\mathbb{P}\left(\frac{\tau_n^\mathsf{rs}}{n^{d/(d+1)}}\leq T-\varepsilon\right)}\\
&=&\mathbb{P}\left(\left.\mathbb{T}_n^d\middle\backslash\mathcal{B}_n^\mathsf{rs}\left(\left\lfloor(T-\varepsilon)\cdot n^{d/(d+1)}\right\rfloor\right)\right.=\emptyset\right)\\
&\leq&\mathbb{P}\left(\left|\frac{\#\left(\mathbb{T}_n^d\middle\backslash\mathcal{B}_n^\mathsf{rs}\left(\left\lfloor(T-\varepsilon)\cdot n^{d/(d+1)}\right\rfloor\right)\right)}{n^d}-\frac{1}{y^{(d)}(T-\varepsilon)}\right|\geq\frac{1}{y^{(d)}(T-\varepsilon)}\right)\underset{n\to\infty}{\longrightarrow}0.
\end{eqnarray*}
Let us now turn to the upper bound, which is more sophisticated.
We will use again the version $\mathcal{B}_n^*(\cdot)$ of the rejection sampling burning process, constructed out of the random variables $\left(X_i^h\,;\,i,h\in\mathbb{N}^*\right)$.
Let us denote by $\tau_n^*=\inf\left\{k\in\mathbb{N}:\mathcal{B}_n^*(k)=\mathbb{T}_n^d\right\}$ the corresponding random burning number.
We will show that there exists a constant $\gamma\in(0,\infty)$ such that for each $\varepsilon\in(0,T)$, we have
\begin{equation}\label{eq:asymptoticgoal}
\mathbb{P}\left(\frac{\tau_n^*}{n^{d/(d+1)}}>T+\gamma\cdot\varepsilon\right)=\mathbb{P}\left(\left.\mathbb{T}_n^d\middle\backslash\mathcal{B}_n^*\left(\left\lfloor(T+\gamma\cdot\varepsilon)\cdot n^{d/(d+1)}\right\rfloor\right)\right.\neq\emptyset\right)\underset{n\to\infty}{\longrightarrow}0.
\end{equation}
Fix $\varepsilon\in(0,T)$.
The idea is to start from \eqref{eq:functionalconvergence}, which tells us that at time $\left\lfloor(T-\varepsilon)\cdot n^{d/(d+1)}\right\rfloor$, with high probability as $n\to\infty$, only a small proportion of vertices are not burned.
We then implement a strategy to burn these vertices in time of order at most a constant times $\varepsilon\cdot n^{d/(d+1)}$ with high probability as $n\to\infty$.
Let us now describe this strategy.
Set ${r_k=\varepsilon\cdot n^{d/(d+1)}\cdot2^{-k}}$ for all ${k\in\mathbb{N}^*}$, and for all $n\in\mathbb{N}^*$ sufficiently large so that ${\left(\varepsilon\cdot n^{d/(d+1)}\right)^{2/3}\geq2}$, let
\[\ell=\left\lfloor\frac{2}{3}\cdot\log_2\left(\varepsilon\cdot n^{d/(d+1)}\right)\right\rfloor\]
be the largest integer $k\in\mathbb{N}^*$ such that ${r_k\geq\left(\varepsilon\cdot n^{d/(d+1)}\right)^{1/3}}$.
Then, consider a collection $\left(A_u\,;\,u\in\partial[\mathbf{T}]_k\right)_{k\in\llbracket1,\ell\rrbracket}$ of nested partitions of $\mathbb{T}_n^d$, indexed by a plane tree $\mathbf{T}$ with height $\ell$ (we denote by $\partial[\mathbf{T}]_k$ the set of nodes in generation $k$ of $\mathbf{T}$), constructed in such a way that the following holds: there exists a constant $C=C(d)\in(1,\infty)$ such that
\begin{itemize}
\item The subsets $\left(A_u\,;\,u\in\partial[\mathbf{T}]_1\right)$ partition $\mathbb{T}_n^d$, have diameter at most $C\cdot r_1$, and have cardinality between $r_1^d$ and $C\cdot r_1^d$,

\item For each $k\in\llbracket1,\ell-1\rrbracket$, and for each node $u\in\partial[\mathbf{T}]_k$, the subsets $\left(A_v\,;\,\text{$v$ child of $u$ in $\mathbf{T}$}\right)$ partition $A_u$, have diameter at most $C\cdot r_{k+1}$, and have cardinality between $r_{k+1}^d$ and $C\cdot r_{k+1}^d$.
\end{itemize}
We present a possible construction of such a collection ${(A_u\,;\,u\in\partial[\mathbf{T}]_k)_{k\in\llbracket1,\ell\rrbracket}}$ in Appendix \ref{sec:appendix}.
For future reference, note that for each ${k\in\llbracket1,\ell-1\rrbracket}$, and for each node $u\in\partial[\mathbf{T}]_k$, we have
\[\#\{\text{$v$ child of $u$ in $\mathbf{T}$}\}\cdot r_{k+1}^d\leq\sum_{\text{$v$ child of $u$ in $\mathbf{T}$}}\#A_v=\#A_u\leq C\cdot r_k^d,\]
hence
\begin{equation}\label{eq:childrenu}
\#\{\text{$v$ child of $u$ in $\mathbf{T}$}\}\leq C\cdot2^d.
\end{equation}

Next, fix a parameter $\alpha\in(0,\infty)$ to be adjusted later, and let $(\theta_2,\ldots,\theta_\ell)$ be a collection of independent random variables, independent of the random variables $\left(X_i^h\,;\,i,h\in\mathbb{N}^*\right)$, such that for each $k\in\llbracket1,\ell-1\rrbracket$, the random variable $\theta_{k+1}$ has Poisson distribution with parameter $\alpha\cdot r_{k+1}$.
We then introduce ``stopping'' times $\tau_1,\ldots,\tau_\ell$ as follows: we set
\[\tau_1=\left\lfloor(T-\varepsilon)\cdot n^{d/(d+1)}\right\rfloor+\lfloor C\cdot r_1\rfloor,\]
and by induction, for each $k\in\llbracket1,\ell-1\rrbracket$ such that $\tau_1,\ldots,\tau_k$ have been constructed, we let
\[\tau_{k+1}=\tau_k+\theta_{k+1}+\lfloor C\cdot r_{k+1}\rfloor.\]
For future reference, note that $\tau_\ell$ is bounded above by $(T-\varepsilon+C\cdot\varepsilon+2\cdot\alpha\cdot\varepsilon)\cdot n^{d/(d+1)}$ with high probability as $n\to\infty$.
Indeed, we have
\[\begin{split}
\tau_\ell&=\tau_1+\sum_{k=1}^{\ell-1}(\tau_{k+1}-\tau_k)\\
&\leq(T-\varepsilon)\cdot n^{d/(d+1)}+C\cdot r_1+\sum_{k=1}^{\ell-1}(\theta_{k+1}+C\cdot r_{k+1})\\
&\leq(T-\varepsilon)\cdot n^{d/(d+1)}+C\cdot\sum_{k\geq1}r_k+\sum_{k=1}^{\ell-1}\theta_{k+1}\\
&=(T-\varepsilon)\cdot n^{d/(d+1)}+C\cdot\varepsilon\cdot n^{d/(d+1)}+\sum_{k=1}^{\ell-1}\theta_{k+1},
\end{split}\]
where $\sum_{k=1}^{\ell-1}\theta_{k+1}$ has distribution $\mathrm{Poisson}\left(\sum_{k=1}^{\ell-1}\alpha\cdot r_{k+1}\right)$.
In particular, the latter random variable is stochastically dominated by a $\mathrm{Poisson}$ random variable with parameter $\alpha\cdot\varepsilon\cdot n^{d/(d+1)}$, and an exponential Markov inequality argument yields the following Chernoff bound (see \cite[Section 2.2]{concentration}):
\[\mathbb{P}\left(\sum_{k=1}^{\ell-1}\theta_{k+1}\geq2\cdot\alpha\cdot\varepsilon\cdot n^{d/(d+1)}\right)\leq\exp\left(-\eta\cdot\alpha\cdot\varepsilon\cdot n^{d/(d+1)}\right),\]
where $\eta=2\ln(2)-1\in(0,\infty)$.
We thus find that
\begin{equation}\label{eq:asymptotictaul}
\mathbb{P}\left(\tau_\ell\geq(T-\varepsilon+C\cdot\varepsilon+2\cdot\alpha\cdot\varepsilon)\cdot n^{d/(d+1)}\right)\leq\exp\left(-\eta\cdot\alpha\cdot\varepsilon\cdot n^{d/(d+1)}\right)\underset{n\to\infty}{\longrightarrow}0,
\end{equation}
as desired.

We will now construct, out of the random variables $\left(X_i^h\,;\,i,h\in\mathbb{N}^*\right)$ and $(\theta_2,\ldots,\theta_\ell)$, a random subtree $\mathcal{T}$ of $\mathbf{T}$, in such a way that
\begin{equation}\label{eq:asymptotictree}
\left.\mathbb{T}_n^d\middle\backslash\mathcal{B}_n^*(\tau_k)\right.\subset\bigcup_{u\in\partial[\mathcal{T}]_k}A_u\quad\text{for all $k\in\llbracket1,\ell\rrbracket$.}
\end{equation}
In view of \eqref{eq:asymptoticgoal} and \eqref{eq:asymptotictaul}, our goal will then be to show that $\#\partial[\mathcal{T}]_\ell$ is small, in order to get that $\#\left(\mathbb{T}_n^d\middle\backslash\mathcal{B}_n^*(\tau_\ell)\right)$ is small (see below \eqref{eq:asymptoticgoaltree}).
To begin with, we set the first generation of $\mathcal{T}$ to be:
\[\partial[\mathcal{T}]_1=\left\{u\in\partial[\mathbf{T}]_1:A_u\cap\mathcal{B}_n^*\left(\left\lfloor(T-\varepsilon)\cdot n^{d/(d+1)}\right\rfloor\right)=\emptyset\right\}.\]
For future reference, note that we have
\[\#\partial[\mathcal{T}]_1\cdot r_1^d\leq\sum_{u\in\partial[\mathcal{T}]_1}\#A_u\leq\#\left(\mathbb{T}_n^d\middle\backslash\mathcal{B}_n^*\left(\left\lfloor(T-\varepsilon)\cdot n^{d/(d+1)}\right\rfloor\right)\right),\]
hence
\begin{equation}\label{eq:treegen1}
\#\partial[\mathcal{T}]_1\leq2^{d+1}\cdot\varepsilon^{-(d+1)}\cdot\frac{\#\left(\mathbb{T}_n^d\middle\backslash\mathcal{B}_n^*\left(\left\lfloor(T-\varepsilon)\cdot n^{d/(d+1)}\right\rfloor\right)\right)}{n^d}\cdot r_1.
\end{equation}
Next, since the subsets $(A_u\,;\,u\in\partial[\mathbf{T}]_1)$ have diameter at most $C\cdot r_1$, by the definition of $\tau_1$, we get
\[\bigcup_{u\in\partial[\mathbf{T}]_1\setminus\partial[\mathcal{T}]_1}A_u\subset\mathcal{B}_n^*(\tau_1).\]
Taking complements in the above display and recalling that the subsets $\left(A_u\,;\,u\in\partial[\mathbf{T}]_1\right)$ partition $\mathbb{T}_n^d$, we get
\[\left.\mathbb{T}_n^d\middle\backslash\mathcal{B}_n^*(\tau_1)\right.\subset\bigcup_{u\in\partial[\mathcal{T}]_1}A_u,\]
as desired.
Next, fix $k\in\llbracket1,\ell-1\rrbracket$, and by induction, assume that generations $1,\ldots,k$ of $\mathcal{T}$ have been constructed, in such a way that
\[\left.\mathbb{T}_n^d\middle\backslash\mathcal{B}_n^*(\tau_j)\right.\subset\bigcup_{u\in\partial[\mathcal{T}]_j}A_u\quad\text{for all $j\in\llbracket1,k\rrbracket$.}\]
We construct generation $(k+1)$ of $\mathcal{T}$ as follows.
For each $i\in\llbracket\tau_k+1,\tau_k+\theta_{k+1}\rrbracket$, we set
\[H_i=\inf\left\{h\in\mathbb{N}^*:X_i^h\in\bigcup_{u\in\partial[\mathcal{T}]_k}A_u\right\}\in\llbracket1,\infty\rrbracket,\]
and we let
\[Z_i=\begin{cases}
X_i^{H_i}&\text{if $H_i<\infty$,}\\
X_i^1&\text{otherwise.}
\end{cases}\]
Note that by construction, we have $Z_i=Y_i^*$ as soon as $Z_i\notin\mathcal{B}_n^*(i-1)$.
Indeed, if $Z_i\notin\mathcal{B}_n^*(i-1)$, then since
\[\left.\mathbb{T}_n^d\middle\backslash\mathcal{B}_n^*(i-1)\right.\subset\left.\mathbb{T}_n^d\middle\backslash\mathcal{B}_n^*(\tau_k)\right.\subset\bigcup_{u\in\partial[\mathcal{T}]_k}A_u,\]
we must have $H_i=H_i^*<\infty$, hence $Z_i=X_i^{H_i}=X_i^{H_i^*}=Y_i^*$.
We then let
\[\partial[\mathcal{T}]_{k+1}=\bigcup_{u\in\partial[\mathcal{T}]_k}\left\{\text{$v$ child of $u$ in $\mathbf{T}$}:A_v\cap\left\{Z_i\,;\,i\in\llbracket\tau_k+1,\tau_k+\theta_{k+1}\rrbracket\right\}=\emptyset\right\}.\]
Now, since the subsets ${(A_v\,;\,v\in\partial[\mathbf{T}]_{k+1})}$ have diameter at most $C\cdot r_{k+1}$, by the definition of $\tau_{k+1}$, we get
\begin{equation}\label{eq:genk+1}
\bigcup_{v\in\partial[\mathbf{T}]_{k+1}\setminus\partial[\mathcal{T}]_{k+1}}A_v\subset\mathcal{B}_n^*(\tau_{k+1}).
\end{equation}
Indeed, fix $v\in\partial[\mathbf{T}]_{k+1}\setminus\partial[\mathcal{T}]_{k+1}$.
By definition, there exists $i\in\llbracket\tau_k+1,\tau_k+\theta_{k+1}\rrbracket$ such that $Z_i\in A_v$.
Then, we have either $Z_i\in\mathcal{B}_n^*(i-1)$, or $Z_i\notin\mathcal{B}_n^*(i-1)$ and in that case $Z_i=Y_i^*\in\mathcal{B}_n^*(i)$.
Either way, we have $Z_i\in\mathcal{B}_n^*(\tau_k+\theta_{k+1})$, and thus $A_v\subset\overline{B}(Z_i,\lfloor C\cdot r_{k+1}\rfloor)\subset\mathcal{B}_n^*(\tau_{k+1})$.
Taking complements in \eqref{eq:genk+1} and recalling that the subsets $\left(A_v\,;\,v\in\partial[\mathbf{T}]_{k+1}\right)$ partition $\mathbb{T}_n^d$, we get
\[\left.\mathbb{T}_n^d\middle\backslash\mathcal{B}_n^*(\tau_{k+1})\right.\subset\bigcup_{v\in\partial[\mathcal{T}]_{k+1}}A_v,\]
which concludes the induction step.

Now that the random subtree $\mathcal{T}$ of $\mathbf{T}$ has been constructed, the rest of the proof will be devoted to showing that there exists a constant $\beta\in(0,\infty)$ such that
\begin{equation}\label{eq:asymptoticgoaltree}
\mathbb{P}\left(\#\partial[\mathcal{T}]_\ell\geq\beta\cdot r_\ell\right)\longrightarrow0\quad\text{as $n\to\infty$.}
\end{equation}
First, taking \eqref{eq:asymptoticgoaltree} for granted, let us prove \eqref{eq:asymptoticgoal}.
By \eqref{eq:asymptotictree}, we have
\[\#\left(\mathbb{T}_n^d\middle\backslash\mathcal{B}_n^*(\tau_\ell)\right)\leq\sum_{u\in\partial[\mathcal{T}]_\ell}\#A_u\leq\#\partial[\mathcal{T}]_\ell\cdot C\cdot r_\ell^d,\]
hence
\[\mathbb{P}\left(\#\left(\mathbb{T}_n^d\middle\backslash\mathcal{B}_n^*(\tau_\ell)\right)\geq\beta\cdot r_\ell\cdot C\cdot r_\ell^d\right)\leq\mathbb{P}\left(\#\partial[\mathcal{T}]_\ell\geq\beta\cdot r_\ell\right)\underset{n\to\infty}{\longrightarrow}0.\]
Now, by the definition of $\ell$, we have
\[\beta\cdot r_\ell\cdot C\cdot r_\ell^d=\beta C\cdot r_\ell^{d+1}\leq\beta C\cdot\left(2\cdot\left(\varepsilon\cdot n^{d/(d+1)}\right)^{1/3}\right)^{d+1}=\beta C\cdot\left(2\cdot\varepsilon^{1/3}\right)^{d+1}\cdot n^{d/3}.\]
Thus, since $n^{d/3}\ll\left(n^{d/(d+1)}\right)^d$ as $n\to\infty$, we get that $\beta C\cdot r_\ell^{d+1}\leq\#\overline{B}\left(0,\left\lfloor\varepsilon\cdot n^{d/(d+1)}\right\rfloor\right)$ for all sufficiently large $n\in\mathbb{N}^*$, and it follows that
\[\mathbb{P}\left(\#\left(\mathbb{T}_n^d\middle\backslash\mathcal{B}_n^*(\tau_\ell)\right)\geq\#\overline{B}\left(0,\left\lfloor\varepsilon\cdot n^{d/(d+1)}\right\rfloor\right)\right)\longrightarrow0\quad\text{as $n\to\infty$.}\]
In turn, this entails that with high probability as $n\to\infty$, we have $d\left(x;\mathcal{B}_n^*(\tau_\ell)\right)\leq2\cdot\varepsilon\cdot n^{d/(d+1)}$ for all $x\in\mathbb{T}_n^d$, hence $\tau_n^*\leq\tau_\ell+2\cdot\varepsilon\cdot n^{d/(d+1)}$.
Combining this bound with \eqref{eq:asymptotictaul}, we obtain \eqref{eq:asymptoticgoal}, with $\gamma=C+2\cdot\alpha+1$.

Finally, to complete the proof of the proposition, it remains to prove \eqref{eq:asymptoticgoaltree}.
Using \eqref{eq:blowuprate}, we set
\[\beta=2^{d+1}\cdot\sup_{\delta\in(0,T)}\left(\delta^{-(d+1)}\cdot\frac{2}{y^{(d)}(T-\delta)}\right)\in(0,\infty).\]
Then, we write
\begin{eqnarray}
\mathbb{P}\left(\#\partial[\mathcal{T}]_\ell\geq\beta\cdot r_\ell\right)&\leq&\mathbb{P}\left(\#\partial[\mathcal{T}]_1\geq\beta\cdot r_1\right)\notag\\
&&+\sum_{k=1}^{\ell-1}\mathbb{P}\left(\text{$0<\#\partial[\mathcal{T}]_k\leq\beta\cdot r_k$ and $\#\partial[\mathcal{T}]_{k+1}\geq\beta\cdot r_{k+1}$}\right).\label{eq:sum}
\end{eqnarray}
The definition of $\beta$ allows us to control the first term: by \eqref{eq:treegen1}, we have
\begin{eqnarray*}
\mathbb{P}\left(\#\partial[\mathcal{T}]_1\geq\beta\cdot r_1\right)&\leq&\mathbb{P}\left(\#\partial[\mathcal{T}]_1\geq2^{d+1}\cdot\varepsilon^{-(d+1)}\cdot\frac{2}{y^{(d)}(T-\varepsilon)}\cdot r_1\right)\\
&\leq&\mathbb{P}\left(\frac{\#\left(\mathbb{T}_n^d\middle\backslash\mathcal{B}_n^*\left(\left\lfloor(T-\varepsilon)\cdot n^{d/(d+1)}\right\rfloor\right)\right)}{n^d}\geq\frac{2}{y^{(d)}(T-\varepsilon)}\right)\underset{n\to\infty}{\longrightarrow}0.
\end{eqnarray*}
Finally, it remains to handle the sum at \eqref{eq:sum}.
To this end, we will show that it is possible to adjust $\alpha$ so that for all $k\in\llbracket1,\ell-1\rrbracket$,
\begin{equation}\label{eq:chernoffcond}
\mathbb{P}\left(\text{$0<\#\partial[\mathcal{T}]_k\leq\beta\cdot r_k$ and $\#\partial[\mathcal{T}]_{k+1}\geq\beta\cdot r_{k+1}$}\right)\leq\exp\left(-\eta\cdot\frac{\beta\cdot r_k}{4}\right),
\end{equation}
where again $\eta=2\ln(2)-1\in(0,\infty)$.
Taking \eqref{eq:chernoffcond} for granted, we obtain, by the definition of $\ell$:
\begin{eqnarray*}
\lefteqn{\sum_{k=1}^{\ell-1}\mathbb{P}\left(\text{$0<\#\partial[\mathcal{T}]_k\leq\beta\cdot r_k$ and $\#\partial[\mathcal{T}]_{k+1}\geq\beta\cdot r_{k+1}$}\right)}\\
&\leq&\frac{2}{3}\cdot\log_2\left(\varepsilon\cdot n^{d/(d+1)}\right)\cdot\exp\left(-\eta\cdot\frac{\beta\cdot\left(\varepsilon\cdot n^{d/(d+1)}\right)^{1/3}}{4}\right)\underset{n\to\infty}{\longrightarrow}0.
\end{eqnarray*}
Therefore, to complete the proof, it remains to establish \eqref{eq:chernoffcond}.
For each $k\in\llbracket1,\ell-1\rrbracket$, let us introduce the ``stopping'' $\sigma$-algebra
\[\mathcal{G}_k=\left\{A\in\mathcal{F}:\text{$A\cap(\tau_k\leq i)\in\mathcal{F}_i\vee\sigma(\theta_2,\ldots,\theta_k)$ for all $i\in\mathbb{N}$}\right\},\]
where we recall that $\mathcal{F}_k$ is the $\sigma$-algebra generated by the random variables
\[\left(X_i^h\,;\,\text{$i\in\llbracket1,k\rrbracket$ and $h\in\mathbb{N}^*$}\right),\]
and $\mathcal{F}$ is the underlying $\sigma$-algebra.
The definition of $\mathcal{G}_k$ is tailored so that:
\begin{itemize}
\item The random subset $\partial[\mathcal{T}]_k\subset\partial[\mathbf{T}]_k$ is $\mathcal{G}_k$-measurable,

\item Conditionally on $\mathcal{G}_k$, the random variables $\left(X_{\tau_k+i}^h\,;\,i,h\in\mathbb{N}^*\right)$ and $\theta_{k+1}$ are independent, and have uniform distribution on $\mathbb{T}_n^d$ and Poisson distribution with parameter $\alpha\cdot r_{k+1}$, respectively.
In particular, conditionally on $\mathcal{G}_k$, the process $\{Z_i\,;\,i\in\llbracket\tau_k+1,\tau_k+\theta_{k+1}\rrbracket\}$ is a Poisson process with intensity $\alpha\cdot r_{k+1}$ on $\bigcup_{u\in\partial[\mathcal{T}]_k}A_u$, provided that $\partial[\mathcal{T}]_k\neq\emptyset$.
\end{itemize}
We omit the details, which are elementary.
It follows from these two points that conditionally on $\mathcal{G}_k$, the random variable
\[\#\partial[\mathcal{T}]_{k+1}=\sum_{u\in\partial[\mathcal{T}]_k}\sum_{\text{$v$ child of $u$ in $\mathbf{T}$}}\mathbf{1}(A_v\cap\left\{Z_i\,;\,i\in\llbracket\tau_k+1,\tau_k+\theta_{k+1}\rrbracket\right\}=\emptyset)\]
has the same distribution as a sum of independent Bernoulli random variables with parameter
\[p_v:=\exp\left(-\alpha\cdot r_{k+1}\cdot\frac{\#A_v}{\sum_{u\in\partial[\mathcal{T}]_k}\#A_u}\right),\]
provided that $\partial[\mathcal{T}]_k\neq\emptyset$.
We deduce that almost surely, on the event $\left(0<\#\partial[\mathcal{T}]_k\leq\beta\cdot r_k\right)$, we have: for all $\lambda\in\mathbb{R}_+^*$,
\begin{eqnarray*}
\mathbb{E}\left[\exp\left(\lambda\cdot\#\partial[\mathcal{T}]_{k+1}\right)\,\middle|\,\mathcal{G}_k\right]&=&\prod_{u\in\partial[\mathcal{T}]_k}\prod_{\text{$v$ child of $u$ in $\mathbf{T}$}}\left(1+p_v\cdot\left(e^\lambda-1\right)\right)\\
&\leq&\exp\left(\sum_{u\in\partial[\mathcal{T}]_k}\sum_{\text{$v$ child of $u$ in $\mathbf{T}$}}p_v\cdot\left(e^\lambda-1\right)\right).
\end{eqnarray*}
Using \eqref{eq:childrenu}, we get that almost surely, on the event $\left(0<\#\partial[\mathcal{T}]_k\leq\beta\cdot r_k\right)$, we have
\[\mathbb{E}\left[\exp\left(\lambda\cdot\#\partial[\mathcal{T}]_{k+1}\right)\,\middle|\,\mathcal{G}_k\right]\leq\exp\left(\beta\cdot r_k\cdot C\cdot2^d\cdot p(\alpha)\cdot\left(e^\lambda-1\right)\right)\quad\text{for all $\lambda\in\mathbb{R}_+^*$,}\]
where
\[p(\alpha)=\exp\left(-\alpha\cdot r_{k+1}\cdot\frac{r_{k+1}^d}{\beta\cdot r_k\cdot C\cdot r_k^d}\right)=\exp\left(-\frac{\alpha}{2^{d+1}\cdot\beta\cdot C}\right).\]
In order to obtain \eqref{eq:chernoffcond}, we then choose $\alpha$ large enough so that ${C\cdot2^d\cdot p(\alpha)\leq1/4}$.
By a conditional exponential Markov inequality argument, we obtain
\begin{eqnarray*}
\lefteqn{\mathbb{P}\left(\text{$0<\#\partial[\mathcal{T}]_k\leq\beta\cdot r_k$ and $\#\partial[\mathcal{T}]_{k+1}\geq\beta\cdot r_{k+1}$}\right)}\\
&=&\mathbb{E}\left[\mathbb{P}\left(\#\partial[\mathcal{T}]_{k+1}\geq\beta\cdot r_{k+1}\,|\,\mathcal{G}_k\right)\cdot\mathbf{1}(0<\#\partial[\mathcal{T}]_k\leq\beta\cdot r_k)\right]\\
&\leq&\mathbb{E}\left[\mathbb{E}[\exp(\lambda\cdot\#\partial[\mathcal{T}]_{k+1})\,|\,\mathcal{G}_k]\cdot\exp\left(-\lambda\cdot\beta\cdot r_{k+1}\right)\cdot\mathbf{1}(0<\#\partial[\mathcal{T}]_k\leq\beta\cdot r_k)\right]\\
&\leq&\exp\left(-\lambda\cdot\beta\cdot r_k/2\right)\cdot\mathbb{E}\left[\exp\left(\frac{\beta\cdot r_k}{4}\cdot\left(e^\lambda-1\right)\right)\cdot\mathbf{1}(0<\#\partial[\mathcal{T}]_k\leq\beta\cdot r_k)\right]\\
&\leq&\exp\left(-\lambda\cdot\beta\cdot r_k/2+\frac{\beta\cdot r_k}{4}\cdot\left(e^\lambda-1\right)\right).
\end{eqnarray*}
The optimal choice $\lambda=\ln(2)$ finally yields \eqref{eq:chernoffcond}, which concludes the proof.
\end{proof}

\bibliographystyle{siam}
\bibliography{biblio.bib}

\appendix

\section{Nested partitions}\label{sec:appendix}

In this appendix, we present a possible construction for the collection $(A_u\,;\,u\in\partial[\mathbf{T}]_k)_{k\in\llbracket1,\ell\rrbracket}$ of nested partitions of $\mathbb{T}_n^d$ used in the proof of \eqref{eq:asymptoticrandomburning}.
Let $m=\lfloor\log_2(n)\rfloor$, so that $2^m\leq n<2^{m+1}$.
We first devise a collection $(\Lambda_u\,;\,u\in\partial[\mathbf{U}]_k)_{k\in\llbracket0,m\rrbracket}$ of nested partitions of $\llbracket0,n-1\rrbracket^d\subset\mathbb{Z}^d$ by boxes $\Lambda_u$, indexed by the nodes $u$ of a $2^d$-ary tree $\mathbf{U}$ with height $m$, in such a way that:
\begin{itemize}
\item $\Lambda_\varnothing=\llbracket0,n-1\rrbracket^d$, where $\varnothing$ denotes the root of $\mathbf{U}$ (this notation comes from the Ulam--Harris labelling),

\item For each $k\in\llbracket0,m-1\rrbracket$, and for each node $u\in\partial[\mathbf{U}]_k$, the boxes $(\Lambda_v\,;\,\text{$v$ child of $u$ in $\mathbf{U}$})$ partition $\Lambda_u$,

\item For each $k\in\llbracket0,m\rrbracket$, and for each node $u\in\partial[\mathbf{U}]_k$, the box $\Lambda_u=\prod_{i=1}^d\llbracket a_i(u),b_i(u)\rrbracket$ has side lengths $2^{m-k}\leq b_i(u)-a_i(u)+1\leq2^{m-k+1}$.
\end{itemize}
We obtain this collection $(\Lambda_u\,;\,u\in\partial[\mathbf{U}]_k)_{k\in\llbracket0,m\rrbracket}$ by recursively splitting boxes with the following procedure: given an integer $l\in\mathbb{N}^*$ and a box $\Lambda=\prod_{i=1}^d\llbracket a_i,b_i\rrbracket$ with side lengths ${2^l\leq b_i-a_i+1\leq2^{l+1}}$, we split $\Lambda$ into $2^d$ boxes, by splitting for each $i\in\llbracket1,d\rrbracket$ the interval $\llbracket a_i,b_i\rrbracket$ into an interval $\llbracket a_i,a_i+\lceil(b_i-a_i+1)/2\rceil-1\rrbracket$ with side length $2^{l-1}\leq\lceil(b_i-a_i+1)/2\rceil\leq2^l$, and an interval $\llbracket a_i+\lceil(b_i-a_i+1)/2\rceil,b_i\rrbracket$ with side length $2^{l-1}\leq\lfloor(b_i-a_i+1)/2\rfloor\leq2^l$.

Finally, we derive the desired collection $(A_u\,;\,u\in\partial[\mathbf{T}]_k)_{k\in\llbracket1,\ell\rrbracket}$ from ${(\Lambda_u\,;\,u\in\partial[\mathbf{U}]_k)_{k\in\llbracket0,m\rrbracket}}$ as follows.
Let us set $h=\left\lfloor\log_2\left((2\varepsilon)^{-1}\cdot n^{1/(d+1)}\right)\right\rfloor$, so that
\begin{equation}\label{eq:defh}
2^{m-h}\geq\frac{n/2}{(2\varepsilon)^{-1}\cdot n^{1/(d+1)}}=\varepsilon\cdot n^{d/(d+1)}\quad\text{and}\quad2^{m-h}\leq\frac{n}{\left.\left((2\varepsilon)^{-1}\cdot n^{1/(d+1)}\right)\middle/2\right.}=4\cdot\varepsilon\cdot n^{d/(d+1)}.
\end{equation}
Then, for all $n\in\mathbb{N}^*$ sufficiently large so that $h\geq0$ and $h+\ell\leq m$, we identify, for each $k\in\llbracket1,\ell\rrbracket$, generation $k$ of $\mathbf{T}$ with generation $h+k$ of $\mathbf{U}$.
Moreover, for each $k\in\llbracket1,\ell\rrbracket$, and for each node $u\in\partial[\mathbf{T}]_k$, we set $A_u$ to be the canonical projection of $\Lambda_u\subset\mathbb{Z}^d$ in $\mathbb{T}_n^d=\left.\mathbb{Z}^d\middle/n\mathbb{Z}^d\right.$.
By construction, the subset $A_u\subset\mathbb{T}_n^d$ has diameter at most $d\cdot2^{m-(h+k)+1}$, and has cardinality between $\left(2^{m-(h+k)}\right)^d$ and $\left(2^{m-(h+k)+1}\right)^d$.
This yields the desired result, since by \eqref{eq:defh} there exists a constant $C=C(d)\in(1,\infty)$ such that $d\cdot2^{m-(h+k)+1}\leq C\cdot r_k$, and $\left(2^{m-(h+k)}\right)^d\geq r_k^d$ and $\left(2^{m-(h+k)+1}\right)^d\leq C\cdot r_k^d$.
In fact, we find that $C=8^d$ works.

\end{document}